\documentclass[11pt]{article}

\usepackage[a4paper]{geometry}
\usepackage{amsmath,amsfonts,amssymb,amsthm}
\usepackage{graphicx,verbatim,mathtools,xcolor}
\usepackage[only,llparenthesis,rrparenthesis]{stmaryrd}
\usepackage{bm}
\definecolor{dblue}{rgb}{0.,0.,0.8}
\usepackage[breaklinks,bookmarks=false]{hyperref}
\hypersetup{
    colorlinks=true,
    linkcolor=blue,
    citecolor=blue,
    urlcolor=blue,
}
\usepackage{nccmath}
\numberwithin{equation}{section}
\theoremstyle{plain}
\newtheorem{theorem}{Theorem}[section]

\newtheorem{lemma}[theorem]{Lemma}
\newtheorem{corollary}[theorem]{Corollary}

\theoremstyle{remark}
\newtheorem{remark}{Remark}[section]
\newtheorem{example}[remark]{Example}

\newcommand{\calI}{\mathcal{I}}

\newcommand{\calP}{\mathcal{P}}
\newcommand{\calR}{\mathcal{R}}

\newcommand{\calT}{\mathcal{T}}

\newcommand{\ver}{{\bm{a}}}

\newcommand{\dd}{\mathrm{d}}
\renewcommand{\dim}{d}
\newcommand{\abs}[1]{\lvert#1\rvert}
\newcommand{\tends}{\rightarrow}
\newcommand{\norm}[1]{\lVert#1\rVert}
\newcommand{\p}{\partial}
\renewcommand{\th}{{h\tau}}

\DeclareMathOperator{\Div}{div}

\DeclareMathOperator{\diam}{diam}
\DeclareMathOperator*{\argmin}{argmin}
\DeclareMathOperator*{\esssup}{ess\,sup}

\newcommand{\pair}[2]{\langle #1,#2 \rangle}
\newcommand{\R}{\mathbb{R}}
\newcommand{\Om}{\Omega}
\newcommand{\om}{\omega}

\newcommand{\DO}{\partial\Om}
\newcommand{\vphi}{\varphi}
\newcommand{\VK}{\mathcal{V}_K}

\newcommand{\RTNK}{\bm{RTN}_{\widetilde{p}}(K)}
\newcommand{\RTNa}{\bm{RTN}_{\widetilde{p}}(\Ta)}

\newcommand{\sh}{\bm{\sigma}_h}
\newcommand{\Ver}{\mathcal{V}}
\newcommand{\Ta}{\calT^{\ver}}
\newcommand{\oma}{{\om_{\ver}}}

\newcommand{\psia}{\psi_{\ver}}

\newcommand{\Qa}{Q_h^{\ver}}
\newcommand{\Va}{\bm{V}_h^\ver}

\newcommand{\stha}{\bm{\sigma}_{\th}^{\ver,n}}

\newcommand{\sth}{\bm{\sigma}_{\th}}

\newcommand{\etaOscY}{\eta_{\mathrm{osc},\tau,Y}}
\newcommand{\etaOscX}{\eta_{\mathrm{osc},\tau,X}}
\newcommand{\etaOscE}{\eta_{\mathrm{osc},\tau,E}}
\newcommand{\etaOscEth}{\eta_{\mathrm{osc},h,\tau,E}}
\newcommand{\wetaOscEth}{\widetilde{\eta}_{\mathrm{osc},h,\tau,E}}

\newcommand{\YT}{Y_T}

\newcommand{\LH}{L^2(0,T;H^1_0(\Om))}
\newcommand{\HHm}{H^1(0,T;H^{-1}(\Om))}
\newcommand{\Yo}{Y_0}

\newcommand{\Bx}{B_X}
\newcommand{\By}{B_Y}
\newcommand{\Bz}{B_Z}

\newcommand{\tXt}{\mathbb{V}_\tau}

\newcommand{\normYT}[1]{\norm{#1}_{Y_T}}
\newcommand{\normYs}[1]{\norm{#1}_{Y,\mathrm{sym}}}

\newcommand{\uha}{\overline{u}_\tau}
\newcommand{\T}{\mathcal{T}}
\newcommand{\VTp}{V_{h}}
\newcommand{\uht}{u_{h,\tau}}
\newcommand{\Uht}{U_{h,\tau}}
\newcommand{\ouht}{\overline{u}_{h,\tau}}

\newcommand{\Vht}{\mathbb{V}_{h,\tau}}
\newcommand{\Vhtp}{\Vht^+}
\newcommand{\sht}{\bm{\sigma}_{h,\tau}}

\newcommand{\etaOscEthGlobal}{\eta_{\mathrm{osc},h,\tau,E,\star}}

\title{An introduction to the \textit{a posteriori} error analysis of parabolic partial differential equations}
\author{Iain~Smears\footnotemark[2]}
\date{}

\begin{document}

\renewcommand{\thefootnote}{\fnsymbol{footnote}}
\footnotetext[2]{Department of Mathematics, University College London, Gower Street, London, WC1E 6BT, United Kingdom.}
\renewcommand{\thefootnote}{\arabic{footnote}}

\maketitle

\begin{abstract}
This article provides a brief introduction to the \emph{a posteriori} error analysis of parabolic partial differential equations, with an emphasis on challenges distinct from those of steady-state problems.
Using the heat equation as a model problem, we examine the crucial influence of the choice of error norm, as well as the choice of notion of reconstruction of the discrete solution, on the analytical properties of the resulting estimators, especially in terms of the efficiency of the estimators.
\end{abstract}

 \smallskip

\tableofcontents

\section{Introduction}\label{sec:intro}

The numerical approximation of partial differential equations (PDE) is crucial in many scientific and engineering fields. 
Ensuring the accuracy of the approximate solutions necessitates robust error control by \emph{a posteriori} error estimators, which provide computable bounds or estimates for the discretization error based on the computed solution and problem data, without requiring knowledge of the exact solution. 
The error estimators are also central ingredients in the design of adaptive algorithms, which aim to achieve gains in computational efficiency and prescribed error tolerances. 
The purpose of this expository article is to present a brief, accessible, and relatively self-contained, introduction to this topic. 
With this purpose in mind, we use as a model problem the heat equation

\begin{equation}\label{eq:parabolic}
\begin{aligned}
\p_t u - \Delta u &= f & & \text{in }\Om\times(0,T),\\
 u &= 0 & &\text{on }\DO\times (0,T),\\
 u(0) &= u_0  & &\text{in }\Om,
\end{aligned}
\end{equation}
where $\Om$ is a bounded open set in $\R^d$, $d\geq 1$, which we assume to have a Lipschitz boundary, and where $T>0$ is the final time. The assumptions on the data will be made precise below.
We do not assume that the reader is familiar with the current research literature.
Instead, we only assume that the reader has some foundations in the analysis of parabolic PDE and Bochner spaces, as can be found in many books, such as~\cite{Evans1998,LionsMagenes1972,Wloka1987}.

The analysis of \emph{a posteriori} error estimators for time-dependent PDE poses a number of challenges that are not typically encountered for their steady-state counterparts. 
There is not only a very wide range of possible choices of numerical methods, both with regards to temporal and spatial discretization, but even for a chosen method, there can be significant additional complications, such as mesh-adaptation in between time-steps~\cite{ChenFeng2004,Dupont1982,ErikssonJohnson1991,GaspozSiebertKreuzerZiegler2019,Kreuzer2012}.  
In addition, there are many approaches and techniques to the derivation and construction of the \emph{a posteriori} error estimators. 
To mention only a few, there are techniques based on dual problems~\cite{ErikssonJohnson1991,JohnsonNieThomee1990}, elliptic reconstructions~\cite{LakkisMakridakis2006,LakkisMakPryer2015}, energy type bounds~\cite{BergamBernardiMghazli2005,Picasso1998}, or inf-sup stability of the problem~\cite{ESV2017,ESV2019,GaspozSiebertKreuzerZiegler2019,Verfurth2003}.
However, the purpose here is not to review the breadth of possibilities in the literature. 
Instead, the focus is on some fundamental aspects that are relevant to essentially all numerical computations and estimation techniques for parabolic problems. 

The first focal point that we consider is the effect of the choice of norm in which to measure the error.
For parabolic problems, there is a plethora of norms in which one can measure the error.
We are particularly interested in those functional settings that are connected to an inf-sup stable formulation of the problem, so that there is an equivalence between the norm of the error and a (usually noncomputable) dual norm of the residual. 
In these cases, the question of efficiency of the estimators seems particularly relevant, yet turns out to be a rather challenging one.
The problem, as we detail below, is that the same estimator (more precisely, the temporal jump estimator) appears in the corresponding upper bounds on the error for various norms of the error, yet it is known from counter-examples (c.f.\ Example~\ref{ex:inefficiency_jump_X_norm} below) that it is not an efficient estimator for all of these norms. 
Furthermore, for some norms where some efficiency results are available, the question of locality of the efficiency bounds is also more complicated than for steady-state problems.
The second issue that we wish to focus on is perhaps more subtle and easy to overlook. For many time-stepping schemes, such as the implicit Euler method, there is an element of choice in how to reconstruct a discrete function that extends the computed values at the time-step points to the whole time interval.
Hence there is some flexibility in the precise notion of numerical solution that is to be compared against the exact solution.
For instance, two popular choices in the case of the implicit Euler method are a piecewise constant-in-time reconstruction, and a continuous piecewise affine-in-time reconstruction.
Yet there is generally no unique answer, for all problem data, as to which provides the smaller error (see Example~\ref{ex:comparison} below).
Furthermore, the efficiency of the \emph{same} estimator can depend crucially on the choice of reconstruction of the numerical solution (c.f.\ Example~\ref{ex:inefficiency_jump_X_norm} below showing how, depending on the problem data, the estimator might vastly overestimate the error for either one of the reconstructions).

Therefore, we hope that this expository work will help our reader to better navigate the landscape regarding the relationship between the error estimators and the various possible choices of norms and also choices of notion of numerical solution.
We also aim to further clarify the issue of efficiency of the estimators, as this is an aspect of the analysis that has been a challenge for many works in the literature.
In order to highlight how many of the central issues result from time-discretization, we first consider the semi-discrete setting where the problem is discretized only with respect to time. Then, to show how the features of the analysis extend to the fully discrete setting, we also consider the fully discrete setting where space is discretized by a conforming finite element method.
In order to make the presentation as brief, accessible and self-contained as possible, it has been necessary to make some selections on choices of topics.
One direction that we do not discuss is mesh adaptation between time-steps, which requires a more technical analysis.
We refer the reader instead to the research literature mentioned above for further details on this matter.
We also do not cover here several other possible norms for the error, such as $L^2(L^2)$, $L^\infty(L^2)$ or $L^\infty(L^\infty)$ norms, and instead refer the reader to \cite{DemlowLakkisMak2009,ErikssonJohnson1987a,ErikssonJohnson1995,JohnsonNieThomee1990,LakkisMakPryer2015,Sutton2020,Verfurth1998b}.
We also do not go into details on other choices of numerical methods, such as other time-stepping methods, especially higher-order methods~\cite{AkrivisMakNoch2006,AkrivisMakridakisNochetto2009,ESV2017,ESV2019,LozinskiPicasso2009,MakridakisNochetto2006,Schotzau2010}, or nonconformity of the spatial discretization~\cite{ErnVohralik2010,GeorgoulisLakkis2011,NicaiseSoualem2005}.
Despite its importance, the choice of estimators to handle the spatial discretization is not the central focus here, so in the fully discrete setting, we consider only the equilibrated flux estimators, since it allows for some essential results to be quoted succinctly from the literature, which will be especially useful for the treatment of the $L^2(H^1)$ and energy norms. 
However, other spatial estimators, such as residual-based estimators, could also be considered.
Finally, we also do not consider extensions to more general parabolic problems, for instance problems with lower-order terms, although we note that the setting of the analysis generalizes via inf-sup stability for some equivalent norms with time-dependent weights (c.f.\ \cite{OsborneSmears2025}).

%
%
%
%
%
%
%
%
%
%
%

\section{Setting and notation}

\subsection{Sobolev and Bochner spaces}

For a nonempty, bounded, open set $\omega \subset \R^\dim$, $\dim\geq 1$, and an integer $m\geq 1$, let $L^p(\omega;\mathbb{R}^m)$ denote the space of Lebesgue measurable vector-fields $v\colon \omega \rightarrow [-\infty,\infty]^m$ such that $\norm{v}_{L^p(\omega,\R^m)}\coloneqq\left(\int_\omega \abs{v}^p\dd x\right)^{\frac{1}{p}}<\infty$. For the case $m=1$, we abbreviate $L^p(\omega)\coloneqq L^p(\omega;\R^m)$, and likewise for its norm.
For the case $p=2$ and general integer $m\geq 1$, we shall use the shorthand notation $\lVert \cdot \rVert_\omega$ and $(\cdot,\cdot)_{\omega}$ to denote respectively the standard norm and inner-product for both scalar- and vector-valued functions on~$\omega$, where the arguments of the norms and inner-product will distinguish between the scalar and vectorial cases.
Let the Hilbert space $H^1(\omega)\coloneqq W^{1,2}(\omega)$ denote the usual Sobolev space of functions in $L^2(\omega)$ with first-order weak partial derivatives also in $L^2(\omega)$.
The space $H^1(\omega)$ is equipped with the norm
\begin{equation}
\norm{v}_{H^1(\omega)}^2 \coloneqq \norm{v}_\omega^2 + \norm{\nabla v}_\omega^2 \quad \forall v \in H^1(\omega),
\end{equation}
where $\nabla v $ denotes the gradient of $v$.
Let $H^1_0(\omega)$ denote the closure of $C^\infty_0(\omega)$ in the space $H^1(\omega)$, where $C^\infty_0(\omega)$ denotes the space of real-valued infinitely differentiable compactly supported functions on $\omega$.
Note that for $\omega$ bounded, the Poincar\'e inequality implies that the mapping $v\mapsto \norm{\nabla v}_{\omega}$ defines an equivalent norm on $H^1_0(\omega)$, see \cite[Corollary 6.31, p.~184]{AdamsFournier2003}. 
Let $H^{-1}(\omega)$ denote the dual space of $H^1_0(\omega)$, with norm
\begin{equation}\label{eq:dual_norm}
\norm{\Phi}_{H^{-1}(\omega)}\coloneqq \sup_{v\in H^1_0(\omega)\setminus\{0\}}\frac{\pair{\Phi}{v}_{H^{-1}(\omega)\times H^1_0(\omega)}}{\norm{\nabla v}_\omega} \quad \forall \Phi\in H^{-1}(\omega),
\end{equation}
where $\pair{\cdot}{\cdot}_{H^{-1}(\omega)\times H^1_0(\omega)}$ is the duality pairing between $H^{-1}(\omega)$ and $H^1_0(\omega)$.
To simplify the notation, for the special case where $\omega=\Omega$ the domain in~\eqref{eq:parabolic}, we shall drop the subscript and simply write $\pair{\cdot}{\cdot}$ to denote the duality pairing between $H^{-1}(\Omega)$ and $H^1_0(\Omega)$.

For a bounded open set $\omega\subset \R^\dim$, the space $L^2(\omega)$ can be canonically embedded into $H^{-1}(\omega)$ through $\pair{w}{v}_{H^{-1}(\omega)\times H^1_0(\omega)} = (w,v)_{\omega}$ for all $v\in H^1_0(\omega)$. The density of $H^1_0(\omega)$ in $L^2(\omega)$ implies the injectivity of the embedding $L^2(\omega)$ into $H^{-1}(\omega)$; i.e.\ $\pair{w}{v}_{H^{-1}(\omega)\times H^1_0(\omega)}=0$ for all $v\in H^1_0(\omega)$ if and only if $w =0$. 
Thus the spaces $H^1_0(\omega)$, $L^2(\omega)$ and $H^{-1}(\omega)$ form a \emph{Gelfand triple}, with $H^1_0(\omega)\subset L^2(\omega)\subset H^{-1}(\omega)$, where each embedding is continuous, compact, dense and injective, see also~\cite[p.~262]{Wloka1987}.

In the following, we frequently consider Bochner spaces consisting of functions that map a bounded open interval~$I\subset \R$ into various Sobolev spaces over some spatial domain~$\omega$. 
Although we state here some of the main definitions of these spaces and mention some basic properties, the reader may find a more comprehensive introduction in \cite[Ch.~5]{Evans1998}, \cite[Ch.4]{Wloka1987}, and in~\cite{Yosida1995}. 
Let~$V$ be a real Banach space. The space $L^p(I;V)$ is the space of all strongly measurable functions $v:I\mapsto V$ such that $\norm{v}_{L^p(0,T;V)}<\infty$, where 
\begin{equation}
\norm{v}_{L^p(I;V)}\coloneqq \begin{cases}
\left(\int_I \norm{v(t)}_{V}^p\mathrm{d}t \right)^\frac{1}{p}, &p\in [1,\infty),
\\ \esssup_{t\in I}\norm{v(t)}_V, &p = \infty.
\end{cases}
\end{equation}
The main examples of Bochner spaces that we will use below are the spaces $L^2(I;H^1_0(\omega))$, $L^2(I;L^2(\omega))$ and $L^2(I;H^{-1}(\omega))$.
We note that, if $1\leq p<\infty$, the spaces $L^p(I\times \omega)$ and $L^p(I;L^p(\omega))$ are isometrically isomorphic to each other; however this does not extend to the case $p=\infty$, see \cite[p.~24]{Roubicek2013}.
The notion of weak derivatives extends to Bochner spaces, see~\cite[p.~285]{Evans1998}: in particular, for $v\in L^p(I;V)$, we say that $v$ has weak time derivative $ w \in L^1(I;V)$, and we write $\p_t v \coloneqq w$, if and only if 
\begin{equation}\label{eq:bochner_weak_derivative}
\int_I  v(t) \p_t \phi(t) \mathrm{d}t = - \int_I  w(t) \phi(t) \mathrm{d}t \quad \forall \phi \in C^\infty_0(I),
\end{equation}
where the integrals in~\eqref{eq:bochner_weak_derivative} are $V$-valued Bochner integrals, see~\cite[p.~650]{Evans1998} and \cite[p.~388]{Wloka1987} for a definition. 
Using the fact that $\pair{\Phi}{\int_I v(t)\mathrm{d}t}_{V^*\times V}=\int_I\pair{\Phi}{v(t)}_{V^*\times V}\mathrm{d}t$ for all $v\in L^1(I;V)$ and $\Phi\in V^*$, c.f.\ \cite[p.~650]{Evans1998}, and the Hahn--Banach Theorem, it can be shown that~\eqref{eq:bochner_weak_derivative} is equivalent to
\begin{equation}\label{eq:bochner_weak_derivative_2}
\int_I  \pair{\Phi}{v(t)}_{V^*\times V} \p_t \phi(t) \mathrm{d}t = - \int_I \pair{\Phi}{w(t)}_{V^*\times V} \phi(t) \mathrm{d}t, \quad \forall \Phi \in V^*,\quad \forall \phi \in C^\infty_0(I).
\end{equation}
The space $H^1(I;V)$ is the space of functions in $v\in L^2(I;V)$ that have weak time derivative $\p_t v \in L^2(I;V)$. 
A norm on $H^1(I;V) $ is given by $\norm{v}_{H^1(I;V)}=\left(\int_I\norm{v}_{V}^2+\norm{\p_t v}_{V}^2\mathrm{d}t\right)^{\frac{1}{2}}$.

\subsection{Function spaces for parabolic PDE}\label{sec:function_spaces_for_parabolic_pde}


We start by defining several function spaces that will be of special use in the analysis.
First, let
\begin{equation}\label{eq:XY_spaces_def_1}
\begin{aligned}
	X & \coloneqq \LH.
\end{aligned}
\end{equation}
Since $\Omega$ is bounded, we shall equip $X$ with the norm $\norm{\cdot}_X$ defined by
\begin{equation}\label{eq:X_norm}
\norm{v}_X^2 \coloneqq \int_0^T \norm{\nabla v}_\Omega^2 \mathrm{d}t \quad\forall v \in X.
\end{equation}
We also let $Z\coloneqq X\times L^2(\Omega) $ denote the product of the spaces $X$ and $L^2(\Omega)$.
A norm on the space $Z$ is defined by
\begin{equation}\label{eq:Z_norm}
\begin{aligned}
\norm{\bm{v}}_Z^2 &\coloneqq \norm{v}_X^2+\frac{1}{2}\norm{v_T}_\Omega^2 && \forall\bm{v}=(v,v_T)\in Z.
\end{aligned}
\end{equation}
Note that the choice of a factor of $\frac{1}{2}$ before the term $\norm{v_T}_\Omega^2$ in~\eqref{eq:Z_norm} is motivated by the \emph{energy} norm on the solution of the heat equation, which we will consider in later sections.
Next, let 
\begin{equation}\label{eq:Y_def}
Y \coloneqq \LH\cap \HHm,
\end{equation}
where, to be clear, the space $Y$ is the space of functions $ \varphi$ such that $\varphi \in X=\LH$ and, after taking the embedding from $H^1_0(\Omega) $ into $H^{-1}(\Omega)$  as described above, then also $\varphi \in H^1(0,T;H^{-1}(\Omega))$. 
A possible choice of norm on $Y$ is given by $\varphi \mapsto \sqrt{\norm{\varphi}^2_X+\norm{\varphi}^2_{H^1(I;H^{-1}(\Omega))}}$, although see~\eqref{eq:Y_norm} below.
It is known that $Y$ is continuously embedded in $C([0,T];L^2(\Omega))$, see~\cite[p.~287]{Evans1998}, and thus we can define a more suitable norm on $Y$ for the purpose of \emph{a posteriori} error analysis for~\eqref{eq:parabolic}, given by
\begin{equation}\label{eq:Y_norm}
\begin{aligned}
\norm{\vphi}_{Y}^2 &\coloneqq \int_0^T \left(\norm{\p_t \vphi}_{H^{-1}(\Om)}^2 + \norm{\nabla \vphi}_\Omega^2 \right) \dd t + \norm{\vphi(T)}_\Omega^2 & & \forall \vphi \in Y.
\end{aligned}
\end{equation}
It will be seen in Section~\ref{sec:infsup_Y} below that the choice of norm~\eqref{eq:Y_norm} leads to optimal constants in the analysis of the inf-sup condition for the bilinear form associated to the heat operator in~\eqref{eq:parabolic}, which is advantageous for the goal of obtaining quantitative bounds on the error.
Note that the final term in the right-hand side of~\eqref{eq:Y_norm} is finite as a result of the the continous embedding $Y\subset C([0,T];L^2(\Omega)) $, and thus $\norm{\cdot}_Y$ is well-defined and is moreover equivalent to 
$\sqrt{\norm{\cdot}^2_X+\norm{\cdot}^2_{H^1(I;H^{-1}(\Omega))}}$.

Moreover, the continuous embedding of $Y$ in $C([0,T];L^2(\Omega))$ also implies that we can define several useful subspaces of $Y$ with conditions at the initial and final times.
In particular, we define the spaces
\begin{equation}\label{eq:Yo_YT_def}
\begin{aligned}
\Yo  \coloneqq \{ \vphi \in Y,\; \vphi(0) = 0\}, &&& \YT  \coloneqq \{ \vphi \in Y,\;  \vphi(T) =0 \},
	\end{aligned}
\end{equation}
which denote respectively the subspaces of function in $Y$ that vanish at $t=0$ and $t=T$.
We leave $\Yo$ equipped with norm $\norm{\cdot}_{Y}$ defined in~\eqref{eq:Y_norm} above, but we shall equip $\YT$ with the norm $\norm{\cdot}_{\YT}$ defined by
\begin{equation}\label{eq:YT_def}
\begin{aligned}
 \norm{\vphi}_{\YT}^2 & \coloneqq \int_0^T \left( \norm{\p_t \vphi}_{H^{-1}(\Om)}^2 + \norm{\nabla \vphi}_\Omega^2 \right)\dd t + \norm{\vphi(0)}_\Omega^2 & & \forall \vphi \in \YT.
 \end{aligned}
\end{equation}
We will also consider the norm $\normYs{\cdot}$ on the space $Y$ defined by
\begin{equation}\label{eq:Ys_def}
\begin{aligned}
\normYs{\varphi} &\coloneqq \int_0^T\left(
\norm{\partial_t \varphi}_{H^{-1}(\Omega)}^2+\norm{\nabla\varphi}_\Omega^2\right)\dd t + \norm{\varphi(0)}_\Omega^2+\norm{\varphi(T)}_\Omega^2 && \forall \varphi \in Y,
\end{aligned}
\end{equation}
The notation $\normYs{\cdot}$ reflects the fact that this norm is invariant with respect to reversal of the time variable, i.e.\ the map $Y\ni\varphi\mapsto \varphi(T-\cdot)$ is an isometry under the norm $\normYs{\cdot}$. 
This invariance property will play an important role later in Section~\ref{sec:hypercircle_theorem} for the analysis of bounds for the energy norm of the error.

%
%
%
%
%
%
%
%
%
%
%

\subsection{The heat equation}

Returning to the heat equation~\eqref{eq:parabolic}, we shall suppose throughout this work that $f\in L^2(0,T;H^{-1}(\Omega))$ and that $u_0\in L^2(\Omega)$.
Then, it is well-known that \eqref{eq:parabolic} admits a unique weak solution $u\in Y$ that solves
\begin{equation}
\pair{\p_t u (t)}{v}+(\nabla u(t),\nabla v)_\Omega = \pair{f(t)}{v} \quad\forall v\in H^1_0(\Omega), \quad \text{for a.e. } t\in (0,T),
\end{equation}
and $u(0)=u_0$, see for instance~\cite{Evans1998,LionsMagenes1972,Wloka1987}.
Recall that $\pair{\cdot}{\cdot}$ denotes the duality pairing between $H^{-1}(\Om)$ and $H^1_0(\Om)$.

\section{Inf-sup stability of the heat equation}\label{sec:inf_sup_stability}

\subsection{Inf-sup stability: $L^2(H^1)\cap H^1(H^{-1})$ norm}\label{sec:infsup_Y}

We first consider the inf-sup stability of the heat equation~\eqref{eq:parabolic} when considering the solution in the space $Y=L^2(0,T;H^1(\Omega))\cap H^1(0,T;H^{-1}(\Omega))$.
We start by defining the bilinear form $\By \colon Y\times X\tends \R$ by
\begin{equation}\label{eq:bilinear_Y}
\By(\vphi,v) \coloneqq \int_0^T \left(\pair{\p_t \vphi}{v} + (\nabla \vphi,\nabla v)_\Omega \right)\dd t,
\end{equation}
where $\vphi\in Y$ and $v\in X$ are arbitrary functions.
Then, the problem~\eqref{eq:parabolic} admits the following weak formulation: find $u \in Y$ such that $u(0)=u_0$ and such that
\begin{equation}\label{eq:Y_formulation}
\begin{aligned}
\By(u,v) = \int_0^T \pair{f}{v} \dd t & && \forall v \in X.
\end{aligned}
\end{equation}
The well-posedness of~\eqref{eq:Y_formulation} is well-known and can be shown by Galerkin's method \cite{Evans1998,Wloka1987}.
The following result states an inf--sup stability result for the bilinear form~$\By$ for the above spaces equipped with their respective norms. 
\begin{theorem}[Inf--sup identity]\label{thm:inf_sup_parabolic_Y}
For every $\vphi \in Y$, we have
\begin{equation}
\begin{aligned}
\norm{\vphi}_{Y}^2 &= \left[\sup_{v \in X\setminus\{0\}} \frac{ \By(\vphi,v) }{\norm{v}_{X}}\right]^2 + \norm{\vphi(0)}_\Omega^2 . \label{eq:infsup_continuous_Y}
\end{aligned}
\end{equation}
\end{theorem}
\begin{proof}
For a fixed $\vphi \in Y$, let $w_*\in X$ be defined by $(\nabla w_*, \nabla v )_\Omega = \pair{\p_t \vphi}{v}$ for all $v\in H^1_0(\Om)$, a.e.\ in $(0,T)$, which implies the identity $\norm{w_*}_X^2 = \int_0^T\norm{\p_t \vphi}_{H^{-1}(\Om)}^2\dd t$.
Furthermore, we have 
\begin{equation}
\By(\vphi,v)=\int_0^T(\nabla(w_*+\vphi),\nabla v)_\Omega\,\dd t = (w_*+\vphi, v)_X,
\end{equation}
where $(\cdot,\cdot)_X$ denotes the inner-product on $X$.
Thus the Cauchy--Schwarz inequality implies that $\sup_{v\in X\setminus\{0\}} \By(\vphi,v)/\norm{v}_X = \norm{w_*+\vphi}_X$.
The desired identity~\eqref{eq:infsup_continuous_Y} is then obtained by expanding the square
\begin{equation}\label{eq:lem_Y_norm_equivalence_2}
\begin{split}
\norm{w_*+\vphi}_X^2
&= \int_0^T \norm{\nabla( w_* + \vphi )}^2_\Omega \,\dd t\\
& = \int_0^T \left(\norm{\nabla w_*}_\Omega^2 + 2(\nabla w_*,\nabla \vphi)_\Omega + \norm{\nabla \vphi}_\Omega^2 \right)\dd t \\
& = \int_0^T \left(\norm{\p_t \vphi}_{H^{-1}(\Om)}^2 + 2\pair{\p_t \vphi}{\vphi} + \norm{\nabla \vphi}_\Omega^2 \right)\dd t \\
&= \norm{\vphi}_Y^2 - \norm{\vphi(0)}_\Omega^2,
\end{split}
\end{equation}
where we note that we have used the identity $\int_{0}^T 2 \pair{\p_t \vphi}{\vphi}\,\dd t = \norm{\vphi(T)}_\Omega^2 - \norm{\vphi(0)}_\Omega^2$.
\end{proof}

Theorem~\ref{thm:inf_sup_parabolic_Y} immediately implies the following identity for the $Y$-norm of the solution of~\eqref{eq:parabolic}
\begin{equation}\label{eq:Y_norm_solution}
\norm{u}_Y^2 = \int_0^T \norm{f}_{H^{-1}(\Omega)}^2\dd t + \norm{u_0}_\Omega^2.
\end{equation}
Indeed, this follows from the fact that $X^*$, the dual space of $X=L^2(0,T;H^1_0(\Omega))$, is isometrically isomorphic to $L^2(0,T;H^{-1}(\Omega))$.

\subsection{Inf-sup stability: $L^2(H^{-1})$ norm}\label{sec:infsup_X}

The heat equation~\eqref{eq:parabolic} can also be given a variational formulation that casts the time derivative onto the test functions. In particular, let the bilinear form $\Bx \colon X\times \YT\tends \R$ be defined by
\begin{equation}
\begin{aligned}
\Bx(v,\vphi) \coloneqq \int_0^T \left[- \pair{\p_t \vphi}{v} + (\nabla v,\nabla \vphi)_\Omega \right]\dd t & && \forall v\in X,\; \vphi \in \YT.
\end{aligned}
\end{equation}
Then, the model problem~\eqref{eq:parabolic} admits the following weak formulation: find $u \in X$ such that
\begin{equation}\label{eq:X_formulation}
\begin{aligned}
\Bx(u,\vphi) = \int_0^T \langle f,\vphi\rangle\,\dd t + (u_0,\vphi(0))_\Omega & && \forall \vphi \in \YT.
\end{aligned}
\end{equation}
Note that in~\eqref{eq:X_formulation}, the initial condition $u(0)=u_0$ is expressed as a natural condition, appearing in \eqref{eq:X_formulation}, rather than as an essential condition imposed by the choice of solution space.
Note also that in general, the weak formulation~\eqref{eq:X_formulation} can be extended to more general source terms $f\in \YT^*$ the dual of $\YT$, however for the sake of simplicity we shall restrict ourselves here to the case $f\in L^2(0,T;H^{-1}(\Omega))$, and thus the solution of~\eqref{eq:X_formulation} is equivalently the solution of~\eqref{eq:Y_formulation}.

\begin{theorem}[Inf--sup identity]\label{thm:inf_sup_X}
For every $v \in X$, we have
\begin{equation}
\begin{aligned}
\norm{v}_{X} & = \sup_{\vphi \in Y_T\setminus\{0\}} \frac{ \Bx(v,\vphi) }{\norm{\vphi}_{Y_T}}.  \label{eq:infsup_X}
\end{aligned}
\end{equation}
\end{theorem}
\begin{proof}
Let $R\colon v \mapsto v(T-\cdot)$ denote the time-reversal mapping for functions $v\in X$. It is clear that $R:X\rightarrow X$ is an isometry with $\norm{R v}_X=\norm{v}_X$ for all $v\in X$, while $R:\YT\rightarrow \Yo$ is an isometry with $\norm{R\vphi}_{Y}=\normYT{\vphi}$ for all $\vphi\in \YT$.
Observe also that $\By(R \vphi,v)=\Bx(R v,\vphi)$ for all $v\in X$ and all $\vphi\in \YT$.
Therefore, Theorem~\ref{thm:inf_sup_parabolic_Y} implies that
\begin{equation}\label{eq:adjoint_inf_sup}
\norm{\vphi}_{\YT}=\norm{R\vphi}_Y = \sup_{v\in X\setminus\{0\}} \frac{\By(R\vphi,v)}{\norm{v}_X}
=\sup_{v\in X\setminus\{0\}}\frac{\Bx(R v,\vphi)}{\norm{v}_X}
= \sup_{\widetilde{v}\in X\setminus\{0\}}\frac{\Bx(\widetilde{v},\vphi)}{\norm{\widetilde{v}}_X}.
\end{equation}
Note that in passing to the last identity in~\eqref{eq:adjoint_inf_sup}, we have simply substituted $\widetilde{v}=Rv$ and used the isometry identity $\norm{Rv}_X=\norm{v}_X$ above.
Next, we immediately obtain from~\eqref{eq:adjoint_inf_sup} the lower bound
\begin{equation}
\begin{aligned}
\norm{v}_X \geq \sup_{\vphi \in \YT\setminus\{0\}} \frac{\Bx(v,\vphi)}{ \norm{\vphi}_{\YT}} &&&\forall v\in X.
\end{aligned}
\end{equation}
To obtain the converse bound, let $\vphi_* \in \YT$ denote the solution of 
$$
\Bx(w,\vphi_*) = \int_0^T (\nabla w, \nabla v)_\Omega \dd t  \quad \forall w \in X.
$$
This problem can simply be seen as a backward-in-time heat equation with final time condition $\vphi_*(T)=0$. Hence, we have $\norm{v}_X^2 =  \Bx(v,\vphi_*)$ and  \eqref{eq:adjoint_inf_sup} implies that $\norm{\vphi_*}_{\YT} = \norm{v}_X$. This immediately shows the upper bound
\begin{equation}
\begin{aligned}
\norm{v}_X \leq \sup_{\vphi \in \YT\setminus\{0\}} \frac{\Bx(v,\vphi)}{\norm{\vphi}_{\YT}} &&&\forall v\in X,
\end{aligned}
\end{equation}
which completes the proof of~\eqref{eq:infsup_X}.
\end{proof}

It follows from Theorem~\ref{thm:inf_sup_X} that the solution $u$ of \eqref{eq:parabolic} satisfies
\begin{equation}
\norm{u}_X = \sup_{\vphi\in \YT\setminus\{0\}} \frac{\int_0^T \pair{f}{\vphi} \dd t + (u_0,\vphi(0))_\Omega}{\normYT{\vphi}}.
\end{equation}

\subsection{Inf-sup stability: energy norm}

A very commonly used approach to the analysis of parabolic problems such as~\eqref{eq:parabolic} is to consider testing the weak formulation of the equation with the solution.
It is well-known that in the case of~\eqref{eq:parabolic}, the solution $u$ satisfies the identity
\begin{equation}\label{eq:energy_identity}
\norm{u}_E^2\coloneqq \frac{1}{2}\norm{u(T)}_\Omega^2 + \int_0^T \norm{\nabla u}_\Omega^2\dd t = \int_0^T \pair{f}{u}\dd t + \frac{1}{2}\norm{u_0}_\Omega^2.
\end{equation}
This leads to stability results in what is often called the \emph{energy} norm $\norm{u}_E$ defined above.
The name \emph{energy norm} here is used in reference to the relation between the concept of energy of the corresponding stationary problem and the principle of testing the equation with the solution.
However, the issue with starting the analysis from~\eqref{eq:energy_identity} is that, in most cases, it is not clear how to obtain sharp bounds on the term $\int_0^T\pair{f}{u}\dd t$ in the right-hand side of~\eqref{eq:energy_identity}. 
To overcome this issue, we present here a different approach that places the energy norm of the solution within the framework of an inf-sup identity that relates the energy norm to a suitable dual norm of a bilinear form.
This leads to sharper \emph{a posteriori} error bounds for numerical approximations as shown in later sections.

To begin, we shall make use of a useful alternative formula for the norm~$\normYs{\cdot}$ that was defined in~\eqref{eq:Ys_def} above.
\begin{lemma}\label{lem:Ys_identity}
We have the identity
\begin{equation}\label{eq:Ys_norm_identity}
\normYs{\varphi}^2 = 2\norm{\varphi(T)}_\Omega^2 + \int_0^T\norm{(\p_t + \Delta) \varphi}_{H^{-1}(\Omega)}^2 \dd t \quad \forall \varphi \in Y.
\end{equation}
\end{lemma}
Note that in~\eqref{eq:Ys_norm_identity}, $\Delta \colon H^1_0(\Omega)\rightarrow H^{-1}(\Omega)$ denotes the Laplacian operator.
\begin{proof}
Let $\varphi\in Y$ be arbitrary and let $R:Y\rightarrow Y$ denote the time-reversal map $R\vphi = \vphi(T-\cdot)$. 
Theorem~\ref{thm:inf_sup_parabolic_Y} implies that
\begin{multline}\label{eq:Ys_norm_identity_1}
\normYT{\vphi}^2 = \norm{R\vphi}_Y^2 = \left[\sup_{v\in X\setminus\{0\}} \frac{\By(R\vphi,v)}{\norm{v}_X}\right]^2 + \norm{(R\vphi)(0)}_\Omega^2 
\\ = \int_0^T \norm{(\p_t +\Delta ) \vphi}_{H^{-1}(\Omega)}^2+\norm{\vphi(T)}_\Omega^2.
\end{multline}
We then obtain~\eqref{eq:Ys_norm_identity} from~\eqref{eq:Ys_norm_identity_1} and from the identity $\normYs{\vphi}^2 = \normYT{\vphi}^2+\norm{\vphi(T)}_\Omega^2$.
\end{proof}

Recall that the space $Z=X\times L^2(\Omega)$ was defined in Section~\ref{sec:function_spaces_for_parabolic_pde} above.
Define the bilinear form $\Bz:Z\times Y\tends \R$ by
\begin{equation}\label{eq:B_bilinear}
\begin{aligned}
\Bz(\bm{v},\varphi)\coloneqq (v_T,\varphi(T))_\Omega+ \int_0^T\left(-\pair{\partial_t \varphi}{v} + (\nabla \varphi, \nabla v )_\Omega \right) \mathrm{d}t  ,
\end{aligned}
\end{equation}
for all $\bm{v}=(v,v_T)\in Z$ and all $\varphi\in Y$.
The following theorem states the inf-sup identities for the bilinear form $\Bz$.
\begin{theorem}[Inf-sup identities]\label{thm:Z_infsup}
We have the identities
\begin{subequations}\label{eq:infsup}
\begin{align}
\sup_{\varphi\in Y\setminus\{0\}}\frac{\Bz(\bm{v},\varphi)}{\normYs{\varphi}} &= \norm{\bm{v}}_Z \quad\forall \bm{v}\in Z, \label{eq:Z_infsup_1} \\
 \sup_{\bm{v}\in Z\setminus\{0\}}\frac{\Bz(\bm{v},\varphi)}{\norm{\bm{v}}_Z} &= \normYs{\varphi} \quad \forall \varphi\in Y. \label{eq:Z_infsup_2}
\end{align}
\end{subequations}
\end{theorem}
\begin{proof}
We start by showing that
\begin{equation}\label{eq:infsup_upper_bound}
\abs{\Bz(\bm{v},\varphi)}\leq \norm{\bm{v}}_Z\normYs{\varphi}\quad \forall \bm{v}\in Z,\;\forall \varphi\in Y.
\end{equation}
It follows from the definition of the $H^{-1}(\Omega)$-norm in~\eqref{eq:dual_norm} that\begin{equation}
\begin{split}
\left\lvert\int_0^T \left(-\pair{\partial_t \varphi}{v} + (\nabla \varphi, \nabla v )_\Omega \right) \dd t\right\rvert
&=\left\lvert\int_0^T -\pair{\p_t \varphi + \Delta \varphi}{v} \dd t\right\rvert
\\ &\leq \left(\int_0^T \norm{\p_t \varphi + \Delta \varphi}_{H^{-1}(\Omega)}^2\dd t\right)^{\frac{1}{2}}\norm{v}_X,
\end{split}
\end{equation}
for any $\varphi\in Y$ and $v\in X$, where we recall that $\norm{v}_X=\left(\int_0^T\norm{\nabla v}_\Omega^2\dd t\right)^{\frac{1}{2}}$.
Hence, the Cauchy--Schwarz inequality applied to the terms in $\Bz(\bm{v},\varphi)$ shows that
\begin{equation}
\begin{split}
\abs{\Bz(\bm{v},\varphi)} &\leq \left(\frac{1}{2}\norm{v_T}_
\Omega^2 + \norm{v}_X^2\right)^{\frac{1}{2}}
\left(2\norm{\varphi(T)}_\Omega^2+\int_0^T\norm{ \partial_t \varphi + 
\Delta \varphi}_{H^{-1}(\Omega)}^2\mathrm{d}t\right)^{\frac{1}{2}}
\\& =\norm{\bm{v}}_Z\normYs{\varphi},
\end{split}
\end{equation}
where we have used Lemma~\ref{lem:Ys_identity} in passing to the second line above. This yields~\eqref{eq:infsup_upper_bound}.

Next, let $\bm{v}=(v,v_T)\in Z$ be arbitrary, and let $\varphi_*\in Y$ be the unique solution of the backward parabolic problem: find $\varphi_* \in Y$ such that $\varphi_*(T)=\frac{1}{2}v_T$ and $\partial_t \varphi_* + \Delta \varphi_* =  \Delta v$ in $(0,T)$. 
Then, it is clear that $\Bz(\bm{v},\varphi)= \frac{1}{2}\norm{v_T}_\Omega^2+\norm{v}_X^2 = \norm{\bm{v}}_Z^2$. Furthermore, using Lemma~\ref{lem:Ys_identity}, we have
\begin{equation}
\normYs{\varphi_*}^2 = 2\norm{\varphi(T)}_\Omega^2 + \int_0^T \norm{\partial_t \varphi+\Delta \varphi}_{H^{-1}(\Omega)}^2\mathrm{d}t = \frac{1}{2}\norm{v_T}_
\Omega^2 + \norm{v}_X^2 = \norm{\bm{v}}_Z^2.
\end{equation}
Thus we obtain \eqref{eq:Z_infsup_1} from the above identities and the upper bound~\eqref{eq:infsup_upper_bound}. To show~\eqref{eq:Z_infsup_2}, we take $\bm{v}_*=(v,v_T)$ with $v=(-\Delta)^{-1}(-\partial_t \varphi) + \varphi \in X$, and $v_T = 2\varphi(T)\in L^2(\Omega)$, and we perform similar computations to find that $\norm{\bm{v}}_Z=\normYs{\varphi}$ and $\Bz(\bm{v},\varphi)=\normYs{\varphi}^2$. 
This shows~\eqref{eq:Z_infsup_2}.
\end{proof}

We now show how to apply Theorem~\ref{thm:Z_infsup} to the heat equation~\eqref{eq:parabolic}. After testing \eqref{eq:parabolic} with a test function $\varphi\in Y$ and integrating-by-parts in time, we see that the solution $u$ solves
\begin{equation}\label{eq:Z_weakform}
(u(T),\varphi(T))_{\Omega} + \int_0^T \left(-\pair{\p_t \varphi}{u} + (\nabla u,\nabla \varphi)_\Omega \right)\dd t = \int_0^T \pair{f}{\varphi}\dd t + (u_0,\varphi(0))_\Omega,
\end{equation}
for all $\varphi\in Y$.
Note that similarly to the weak formulation~\eqref{eq:X_formulation}, the formulation~\eqref{eq:Z_weakform} involves casting the time derivative onto a test function; however in~\eqref{eq:Z_weakform} we allow $\varphi \in Y\setminus \YT$, i.e.\ we do not require the test functions to vanish at the final time $T$.
Upon defining $\bm{u}\coloneqq(u,u(T))\in Z$, it follows that~\eqref{eq:Z_weakform} can be written equivalently as $\Bz(\bm{u},\varphi)=\int_0^T \pair{f}{\varphi}\dd t + (u_0,\varphi(0))_\Omega$ for all $\varphi\in Y$. 
Hence~\eqref{eq:Z_infsup_1} implies that
\begin{equation}\label{eq:Z_norm_solution}
\begin{split}
\norm{u}_{E}^2 &\coloneqq \frac{1}{2}\norm{u(T)}_\Omega^2 + \int_0^T \norm{\nabla u}_\Omega^2\dd t
 \\ &= \norm{\bm{u}}_Z^2 = \left[\sup_{\varphi\in Y\setminus\{0\}}\frac{\int_0^T\pair{f}{\varphi}\dd t + (u_0,\varphi(0))_\Omega}{\normYs{\varphi}}\right]^2.
\end{split} 
\end{equation}
This gives an alternative characterization of the energy norm in comparison to~\eqref{eq:energy_identity} above.
As is clear from the proof of Theorem~\ref{thm:Z_infsup}, the function~$\varphi$ that achieves the supremum on the right-hand side of \eqref{eq:Z_norm_solution} is generally not the same as $u$, which offers some explanation as to difference between using~\eqref{eq:Z_norm_solution} over \eqref{eq:energy_identity} as the starting point for the analysis of the energy norm.

\begin{remark}[Bibliographical remarks]
The analysis of the heat equation in terms of inf-sup stability has appeared in various forms in~\cite{ErnGuermond2004,ESV2017,ESV2019,SchwabStevenson2009,TantardiniVeeser2016,UrbanPatera2012}. 
To the best of our knowledge, the analysis of the energy norm presented above is original.
\end{remark}

\section{Discretization in time}\label{sec:discretization_in_time}

In this section we first consider the effects of the temporal discretization of \eqref{eq:parabolic}, without any spatial discretization. This leads us to considering the \emph{a posteriori} error analysis for a \emph{semi-discrete} method.
Our motivation for considering the temporal discretization first is to highlight some key aspects of the analysis that can be expected to hold regardless of any later choice of spatial discretization. 
In this section we shall temporarily assume that $u_0\in H^1_0(\Omega)$ in order to avoid some technical matters unrelated to the main ideas; this assumption will be later removed in the following sections once spatial discretization is also introduced.

\subsection{Implicit Euler Discretization}\label{sec:IE_Discretization}

\paragraph{The implicit Euler method as a finite difference method.}
The implicit Euler method is frequently derived from a finite difference approximation of the time derivative term in~\eqref{eq:parabolic}.
From this point of view, the discretization is introduced as follows: let $\{t_n\}_{n=0}^N$, for some integer $N\geq 1$, denote a strictly increasing sequence of time-step points, with $t_0 = 0$ and $t_N = T$.
Let $\tau_n = t_n - t_{n-1} >0$ denote the time-step length.

For each $n\in\{1,\dots,N\}$, let $H^1_0(\Omega)\ni u_n \approx u(t_n)$ be defined by 
\begin{equation}\label{eq:IE_as_FD}
\left( \frac{u_{n}-u_{n-1}}{\tau_n} , v \right)_{\Omega} + (\nabla u_n, \nabla v )_\Omega = \pair{ f_n}{v} \quad \forall v \in H^1_0(\Omega),
\end{equation}
where $u_0$ is the initial datum from~\eqref{eq:parabolic}, and where $H^{-1}(\Omega)\ni f_n \approx f(t_n)$ is some approximation of the source term which we discuss shortly below. 
Note that the existence and uniqueness of the approximations $u_n \in H^1_0(\Omega)$, $n\in \{1,\dots, N\}$, follows easily from the Lax--Milgram Theorem applied to a sequence of elliptic variational problems in the space $H^1_0(\Omega)$.

\begin{remark}[Choice of $f_n$]
In the following analysis, we allow for some flexibility in the choice of $f_n$.
A traditional choice for $f_n$ is simply the time-point value $f(t_n)$, however this might not be well-defined for general $f\in L^2(0,T;H^{-1}(\Omega))$ due to possible discontinuities in time.
A more general alternative is to let $f_n$ be the mean-value of $f$ over the time-interval $(t_{n-1},t_n)$, i.e.\
\begin{equation}\label{eq:time_mean_value}
f_n \coloneqq \frac{1}{\tau_n} \int_{t_{n-1}}^{t_n} f(t) \mathrm{d}t \in H^{-1}(\Omega).
\end{equation}
Observe that $f_n$ is thus well-defined for any $f\in L^2(0,T;H^{-1}(\Omega))$.
We do not require that $f_n$ to be defined by~\eqref{eq:time_mean_value} for all of the following results, although sharper results can sometimes be obtained specifically for the case where~\eqref{eq:time_mean_value} holds.
\end{remark}

When viewed in the form~\eqref{eq:IE_as_FD}, the implicit Euler method defines approximations of the solution at the time-points $\{t_n\}_{n=0}^N$.
However, for the purposes of \emph{a posteriori} error analysis, the values of the time-point values $\{u_n\}_{n=0}^N$ of the approximation must be extended to a function on the whole time interval $(0,T)$.
Indeed, the relevant norms for the analysis, which were introduced in the previous sections, require the error to be defined at least at almost every point in the time interval $(0,T)$. 
For the case of the implicit Euler method, there are at least two natural, and closely related, approaches to the construction of such functions from the time-point values of the approximations.
We consider each approach in turn.

\paragraph{Piecewise constant in time reconstruction.}

Given the time-points $\{t_n\}_{n=0}^N$, let $I_n\coloneqq (t_{n-1},t_n)$ denote the associated $n$-th time-interval. 
Note that $\{\overline{I_n}\}_{n=1}^N$ is a partition of $[0,T]$.
Let $\mathbb{V}_\tau^+$ denote the space of functions $v\colon [0,T]\rightarrow H^1_0(\Omega)$ that are left-continuous, and also piecewise constant with respect to the time partition. Note that functions $v\in \mathbb{V}_\tau^+ $ have a well-defined value $v(t)\in H^1_0(\Omega)$ for all $t\in [0,T]$, and moreover functions that agree a.e.\ in $(0,T)$ are not identified. 
The notation $\mathbb{V}_\tau^+$ reflects the fact that we are considering here a forward evolution problem.
Then, given the approximation $\{u_{n}\}_{n=0}^N$ defined by~\eqref{eq:IE_as_FD}, 
we define $u_\tau \in \mathbb{V}_\tau^+$ as the unique left-continuous piecewise constant function that equals $u_n$ on the time interval $I_n$, and satisfies in addition $u(0)=u_0$.
Note that $u_\tau \in X$, but in general $u_\tau$ is discontinuous so $u_\tau\not\in Y$.
Also, the left-continuity of $u_\tau$ ensures that $u_\tau(t_n)=u_n$ for each all time $n\in \{0,\dots,N\}$, i.e.\ $u_\tau$ agrees with the time-point values $\{u_{n}\}_{n=0}^N$ at the time-points $\{t_n\}_{n=0}^N$.

\paragraph{Continuous piecewise affine in time reconstruction.}

Alternatively, we can construct a continuous piecewise affine-in-time function from the time-point values. 
Let $U_\tau$ be the unique continuous function that is piecewise affine in time with respect to the time intervals $\{I_n\}_{n=1}^N$ and that equals $u_{n}$ at time $t_n$, for all $n=0,\dots,N$. In other words, $U_\tau$ satisfies
\[
U_\tau (t) = \frac{t-t_{n-1}}{\tau_n} u_n + \frac{t_n-t}{\tau_n} u_{n-1}\quad \forall t\in \overline{I_n}.
\]
Note that in particular, $U_\tau (0) = u_0$.
Owing to the simplifying assumption that $u_0 \in H^1_0(\Omega)$, it follows that $U_\tau(t) \in H^1_0(\Omega)$ for all $t\in [0,T]$. 
Since $U_\tau$ is continuous, it follows that $U_\tau \in H^1(0,T;H^1_0(\Omega)) \subset Y$. 
Note that the weak time-derivative of $U_\tau$ is piecewise constant with respect to the time intervals $\{I_n\}_{n=1}^N$, and thus $\partial_t U_\tau \in \tXt$ and we have
\begin{equation}\label{eq:time_reconstruction_def}
\partial_t U_\tau(t)  = \frac{u_n-u_{n-1}}{\tau_n} \quad \forall t\in (t_{n-1},t_n),\; \forall n\in\{1,\dots,N\}.
\end{equation}
Note that the two reconstructions $u_\tau$ and $U_\tau$ can each be recovered from the other: first, we have $u_\tau(t_n) = U_\tau(t_n)$ for all $n\in \{0,\dots,N\}$, and $U_\tau$ can be obtained from $u_\tau$ and from $u_0$ by interpolation.
The following identity for the difference will be useful later
\begin{equation}\label{eq:jump_difference_equation}
 u_\tau(t) - U_\tau(t) = \frac{t_n-t}{\tau_n}(u_n-u_{n-1}) \quad \forall t\in (t_{n-1},t_n], \quad \forall n\in\{1,\dots,N\}.
\end{equation}

\paragraph{The implicit Euler method: variational formulation.}\label{sec:IE_inconsistent}

We now show that the implicit Euler method~\eqref{eq:IE_as_FD} can be seen as a discretization of~\eqref{eq:Y_formulation}, involving both the piecewise constant reconstruction $u_\tau$ and the continuous piecewise affine reconstruction $U_\tau$.
Let $f_\tau \in L^2(0,T;H^{-1}(\Omega))$ be defined by $f_\tau(t) = f_n $ for all $t\in (t_{n-1},t_n)$, $n\in \{1,\dots,N\}$. 
Thus $f_\tau$ is piecewise constant in time, similar to functions in $\tXt$.
With $f_\tau$ thusly defined, we see that~\eqref{eq:IE_as_FD} is equivalent to
\begin{equation}\label{eq:IE_1}
(\partial_t U_\tau (t) , v )_\Omega + (\nabla u_\tau (t),\nabla v)_\Omega = \pair{f_\tau(t)}{v} \quad \forall v\in H^1_0(\Omega), \forall t\in I_n,
 \end{equation}
for each $n=1,\dots,N$.
Since $\partial_t U_\tau$, $u_\tau$ and $f_\tau$ are all also piecewise constant in time, we then see that~\eqref{eq:IE_1} is equivalent to the global-in-time form 
\begin{equation}\label{eq:semidiscrete_identity_1}
\int_0^T \left( (\partial U_{\tau},v)_\Omega + (\nabla u_\tau,\nabla v)_\Omega \right)\mathrm{d}t = \int_0^T \pair{f_\tau}{v} \mathrm{d}t \quad \forall v \in \tXt,
\end{equation}
where $\tXt $ denotes the space of $H^1_0(\Omega)$-valued functions that are piecewise constant with respect to $\{I_n\}_{n=1}^N$, i.e.\
\begin{equation}
\tXt \coloneqq \left\{ v\in L^2(0,T;H^1_0(\Omega))\colon \; v|_{I_n} \in \mathcal{P}_0(I_n;H^1_0(\Omega)) \quad \forall n\in\{1,\dots,N\} \right\},
\end{equation}
where $\mathcal{P}_0(I_n;H^1_0(\Omega))$ denotes the space of piecewise constant functions in time with values in $H^1_0(\Omega)$.
The global-in-time form~\eqref{eq:semidiscrete_identity_1} should then be compared with the continuous problem~\eqref{eq:Y_formulation}. 
We can either view~\eqref{eq:semidiscrete_identity_1} as a conforming \emph{but inconsistent} method for the continuous piecewise affine approximation $U_\tau$, or as a \emph{nonconforming} method for the piecewise constant and generally \emph{discontinuous} approximation $u_\tau$. 
As we aim to make clear in this tutorial, whichever point of view one takes, the inconsistency/nonconformity of the temporal discretization is the source of many of the challenges for the analysis of \emph{a posteriori} error bounds.

\paragraph{Comparison of the reconstructions.}
It is natural to ask which of two reconstructions $u_\tau$ and $U_\tau$ provides the better approximation of the exact solution $u$. 
This turns out to be a subtle question, even for comparing these approximations only in the norm of $X=L^2(0,T;H^1_0(\Omega))$.
On the one hand, supposing that a problem is fixed where $u\in H^1(0,T;H^1_0(\Omega))$ and assuming the data $f$ is sufficiently well-approximated, e.g.\ if~\eqref{eq:time_mean_value} holds, then it is known from~\emph{a priori} error analysis that $\norm{u-u_\tau}_X$ and $\norm{u-U_\tau}_X$ both have an overall first-order rate of convergence with respect to the time-step sizes, which is known to be sharp from computational experiments. Thus, in many cases, both yield the same rate of convergence in the small time-step limit.
On the other hand, such results from \emph{a priori} error analysis are by nature asymptotic, and do not answer the question for any given fixed temporal discretization. If instead, we fix the time-step size, and vary the problem to be solved, then it turns out that there can be significant differences between 
$\norm{u-u_\tau}_X$ and $\norm{u-U_\tau}_X$.
We illustrate this issue concretely with the following example.

\begin{example}\label{ex:comparison}
Consider the ordinary differential equation $u^\prime + \lambda u = 1$ on $(0,1)$, with $u(0)=0$, where $\lambda >0$. The exact solution is then $u(t)=\frac{1-\mathrm{e}^{-\lambda t}}{\lambda}$ for all $t\in [0,1]$.
This problem can be thought of as arising from a parabolic PDE after a transformation to Fourier modes, and suitable scaling of the parameters.
The solution $u$ is approximated by the implicit Euler method using as single time-step of length $1$.
Thus $u_\tau(t) = \frac{1}{1+\lambda}$ and $U_\tau(t) = \frac{t}{1+\lambda}$ for all $t\in [0,1]$. 
Then, direct calculations show that
\begin{equation}\label{eq:ex_asymptot_1}
\lim_{\lambda \tends \infty } \frac{\norm{u-u_\tau}_{L^2(0,1)}}{\norm{u-U_\tau}_{L^2(0,1)}} =  0,
\end{equation}
so $u_\tau$ is asymptotically a better approximation of $u$ than $U_\tau$ for large $\lambda$. However, for small $\lambda$, the situation is reversed since
\begin{equation}\label{eq:ex_asymptot_2}
\lim_{\lambda \tends 0} \frac{\norm{u-U_\tau}_{L^2(0,1)}}{\norm{u-u_\tau}_{L^2(0,1)}} = 0,
\end{equation}
so $U_\tau$ is asymptotically a better approximation of $u$ than $u_\tau$ for small $\lambda$.
\end{example}

The conclusion of Example~\ref{ex:comparison} is that there is generally no definitive answer to the question of comparing these two approximations, since either $u_\tau$ or $U_\tau$ might be the better approximation, depending on the problem data.
Furthermore, it follows from~\eqref{eq:ex_asymptot_1} and~\eqref{eq:ex_asymptot_2} that it is not possible to bound either $\norm{u-u_\tau}_X$ or $\norm{u-U_\tau}_X$ in terms of the other with constants that are robust with respect to parameters of the problem and the discretization.
We will see later in Section~\ref{sec:X-norm-bounds} that this issue has important repercussions concerning the efficiency of the error estimators when comparing the \emph{a posteriori} error analysis in the $X$ norm.

\subsection{A posteriori error analysis in the $L^2(H^1)\cap H^1(H^{-1})$-norm}\label{sec:Y-norm-bounds}

Since $U_\tau$ defined in~\eqref{eq:time_reconstruction_def} belongs to the space $Y$, it is natural to consider the error $\norm{u-U_\tau}_Y$.
The analysis that follows reveals that the error $\norm{u-U_\tau}_Y$ is closely related to the \emph{temporal jump estimator} $\eta_J$ defined by
\begin{equation}\label{eq:jump_estimator_semidiscrete}
\eta_J \coloneqq \norm{u_\tau - U_\tau}_X = \left(\sum_{n=1}^N \frac{\tau_n}{3} \norm{\nabla  (u_n-u_{n-1}) }^2_\Omega \right)^{\frac{1}{2}}.
\end{equation}
Note that the second identity in~\eqref{eq:jump_estimator_semidiscrete} above follows from~\eqref{eq:jump_difference_equation}.
The name of jump estimator derives from the fact that $\eta_J$ measures the difference in the $X$-norm between the generally discontinuous in time function $u_\tau$ with the continuous in time function $U_\tau$, with the difference determined by the temporal jump $u_n-u_{n-1}$ occurring at the time-step point $t_{n-1}$, for each $n=1\dots,N$.
 
\begin{theorem}\label{thm:apost_semidisc_Y_norm}
Suppose that $u_0\in H^1_0(\Omega)$. 
Then
\begin{align}
& \norm{u-U_\tau}_Y \leq \eta_J + \etaOscY ,\label{eq:semidisc_upper_bound_jump}
\\ & \eta_J \leq \norm{u-U_\tau}_Y + \etaOscY,\label{eq:semidisc_lower_bound_jump}
\end{align}
where the data oscillation $\etaOscY$ is defined by $\etaOscY\coloneqq \norm{f-f_\tau}_{L^2(0,T;H^{-1}(\Omega))} $.
\end{theorem} 
\begin{proof}
We start by defining the residual $\calR_Y\colon Y\tends X^*$ defined by
 \begin{equation}\label{eq:R_Y_def}
 \begin{aligned}
\pair{\calR_Y(\vphi)}{v}_{X^*\times X} \coloneqq \By(u-\vphi,v) &&& \forall \vphi \in Y,\;\forall v \in X,
\end{aligned}
\end{equation}
where $\By$ is defined in~\eqref{eq:bilinear_Y} above and where $\pair{\cdot}{\cdot}_{X^*\times X}$ denotes the duality pairing between the dual space $X^*$ and $X$.
The dual norm of the residuals is naturally defined by $\norm{\calR_Y(\vphi)}_{X^*}\coloneqq \sup_{v\in X\setminus\{0\}} \tfrac{\pair{\calR_Y(\vphi)}{v}}{\norm{v}_X}$.
Theorem~\ref{thm:inf_sup_parabolic_Y} implies the following equivalence between the error and dual norm of the residual: 
\begin{equation} 
\begin{aligned}
\norm{u-\vphi}_{Y}^2 = \norm{\calR_Y(\vphi)}_{X^*}^2 + \norm{u_0-\vphi(0) }_\Omega^2 &&&\forall \vphi \in Y.
\end{aligned}
 \label{eq:Y_error_residual_equivalence}
\end{equation}
Recalling that $U_\tau(0)=u_0=u(0)$, the equivalence between error and residual in~\eqref{eq:Y_error_residual_equivalence} then implies that
\begin{equation}\label{eq:Y-norm_residual}
\norm{u-U_\tau}_Y^2 = \norm{\calR_Y(U_\tau)}_{X^*}^2 + \norm{u_0-U_\tau(0)}_\Omega^2= \norm{\calR_Y(U_\tau)}_{X^*}^2.
\end{equation}
We compute the residual $\calR_Y(U_\tau)$ to find that
\begin{multline}\label{eq:residual_computation}
\pair{\calR_Y(U_\tau)}{v}_{X^*\times X} =\int_0^T \left[  \pair{f}{v} -\pair{\p_t U_\tau}{v}-(\nabla U_\tau,\nabla v)_\Omega \right]\mathrm{d}t
\\ = \int_0^T  \left[\pair{f-f_\tau}{v}  + (\nabla (u_\tau - U_\tau),\nabla v)_\Omega\right] \mathrm{d}t
  + \int_{0}^T \underbrace{\left[\pair{f_\tau}{v} -\pair{\p_t U_\tau}{v}-(\nabla u_\tau,\nabla v)_\Omega \right]}_{=0 \text{ by } \eqref{eq:IE_1}} \mathrm{d}t 
\\  = \int_0^T \pair{f-f_\tau}{v} \mathrm{d}t+ \int_0^T (\nabla (u_\tau -  U_\tau),\nabla v)_\Omega\mathrm{d}t \quad \forall v \in X.
\end{multline}
Then, recalling that $X=L^2(0,T;H^1_0(\Omega))$, we observe that $\int_0^T (\nabla (u_\tau -  U_\tau),\nabla v)_\Omega\mathrm{d}t$ is the inner-product in $X$ between $u_\tau - U_\tau$ and the test function $v\in X$.
Therefore, we apply the Cauchy--Schwarz inequality to the last line of~\eqref{eq:residual_computation} to find that $\norm{R_Y(U_\tau)}_{X^*}\leq \eta_J+\etaOscY$.
Combined with~\eqref{eq:Y-norm_residual}, this implies the upper bound~\eqref{eq:semidisc_upper_bound_jump}.
To prove~\eqref{eq:semidisc_lower_bound_jump}, let $v = u_\tau-U_\tau$, noting that $v\in X$ owing to the hypothesis $u_0\in H^1_0(\Omega)$.
Then, using~\eqref{eq:residual_computation}, we have $\norm{u_\tau - U_\tau}_X^2 = \pair{R_Y(U_\tau)}{v}_{X^*\times X}-\int_0^T\pair{f-f_\tau}{v}\mathrm{d}t$, which implies that $\norm{u_\tau-U_\tau}_X \leq \norm{R_Y(U_\tau)}_{X^*}+\etaOscY$.
This yields~\eqref{eq:semidisc_lower_bound_jump} after using~\eqref{eq:Y-norm_residual}.
\end{proof}

\begin{remark}[Case of vanishing data oscillation]
Theorem~\ref{thm:apost_semidisc_Y_norm} immediately implies that in the case of no data oscillation, i.e.\ if $f=f_\tau$, then we have the identity
\begin{equation}\label{eq:semidiscrete_apost_identity}
\norm{u-U_\tau}_Y = \eta_J.
\end{equation}
In other words, the estimator $\eta_J$ is then exact for the $Y$-norm error.
As we shall see in the following sections, this connection further remains important also in the fully discrete setting, although additional estimators arising from the spatial discretization will also appear.
\end{remark}

\begin{remark}[Local-in-time lower bounds]
In order to keep the presentation simple, we have shown the efficiency bound~\eqref{eq:semidisc_upper_bound_jump} as a global-in-time bound, i.e.\ over the whole time-interval $(0,T)$. 
However, it is possible to show that the efficiency is also local in time, in particular
\begin{equation}\label{eq:time_local_lower_bound_jump}
\int_{I_n} \norm{\nabla (u_\tau-U_\tau)}_\Omega^2\mathrm{d}t \leq \int_{I_n}\left[\norm{\p_t (u-U_\tau)}_{H^{-1}(\Omega)} + \norm{\nabla(u-U_\tau)}_\Omega + \norm{f-f_\tau}_{H^{-1}(\Omega)} \right]^2\mathrm{d}t.
\end{equation}
for each $n\in\{1,\dots,N\}$.
We leave the proof of~\eqref{eq:time_local_lower_bound_jump} as an exercise to the reader, see also \cite{ESV2017}.
\end{remark}

Since Theorem~\ref{thm:apost_semidisc_Y_norm} allows for general choices of $f_\tau$, the data oscillation term~$\norm{f-f_\tau}_{L^2(0,T;H^{-1}(\Omega)}$ appears on the right-hand side of both~\eqref{eq:semidisc_upper_bound_jump} and \eqref{eq:semidisc_lower_bound_jump}.
In particular, this leaves open the possibility that the jump estimator may significantly overestimate the error in the case where the data oscillation is dominant in these bounds.
However, sharper results can be shown in the special case where $f_\tau$ is given as the temporal mean-value of $f$, c.f.~\eqref{eq:time_mean_value}, as shown in Theorem~\ref{thm:semidisc_data_osc_sharp} below.

\begin{theorem}[Error-dominated oscillation]\label{thm:semidisc_data_osc_sharp}
Suppose that $u_0\in H^1_0(\Omega)$ and that $f_\tau$ is given by~\eqref{eq:time_mean_value}. Then
\begin{equation}\label{eq:semidisc_lower_dataosc_proof}
\int_{I_n} \norm{\nabla (u_\tau - U_\tau)}_\Omega^2 + \norm{f-f_\tau}_{H^{-1}(\Omega)}^2\mathrm{d}t \leq C \int_{I_n} \left(\norm{\p_t (u-U_\tau)}_{H^{-1}(\Omega)}^2 + \norm{\nabla (u-U_\tau)}_\Omega^2\right)\mathrm{d}t,
\end{equation}
for each $n\in\{1,\dots,N\}$, where $C$ is some constant independent of any other quantities. 
\end{theorem}
\begin{proof}
We sketch the proof since it follows some of the ideas in~\cite[Proof of Theorem~5.1]{ESV2017}.
The main idea is to choose the semi-discrete test function $v\in \tXt$ in~\eqref{eq:semidiscrete_identity_1} to be given by $v|_{I_n} \coloneqq u_n-u_{n-1}$ and $v$ vanishing on $(0,T)\setminus I_n$.
Note that $u_0\in H^1_0(\Omega)$ implies that $v\in X$ in the case $n=1$.
Also, observe that $\int_{I_n}\norm{\nabla v}_\Omega^2\mathrm{d}t = 3 \int_{I_n}\norm{\nabla(u_\tau-U_\tau)}_\Omega^2\mathrm{d}t$ as a consequence of~\eqref{eq:jump_difference_equation}.
Then, the definition of $f_\tau$ in \eqref{eq:time_mean_value} and the fact that $v\in \tXt$ is constant-in-time implies that $\int_0^T(f_\tau,v)_\Omega\mathrm{d}t=\int_0^T\pair{f}{v}\mathrm{d}t$. 
Thus, we obtain from~\eqref{eq:semidiscrete_identity_1} that
\begin{equation}
\begin{aligned}
\int_{I_n} \norm{\nabla (u_\tau - U_\tau)}_\Omega^2 \mathrm{d}t&= \frac{2}{3}\int_{I_n}(\nabla (u_\tau-U_\tau),\nabla v)_\Omega\mathrm{d}t 
\\& = \frac{2}{3}\int_{I_n} \left[\pair{f}{v} - (\p_t U_\tau,v)_\Omega - (\nabla U_\tau,\nabla v)_\Omega \right] \mathrm{d}t
\\ &= \frac{2}{3} \int_{I_n} \left[\pair{\p_t (u-U_\tau)}{v}  + (\nabla(u-U_\tau),\nabla v)_\Omega\right]\mathrm{d}t.
\end{aligned}
\end{equation}
This shows that $\int_{I_n} \norm{\nabla (u_\tau - U_\tau)}_\Omega^2\mathrm{d}t$ is bounded by the right-hand side of~\eqref{eq:semidisc_lower_dataosc_proof}.
Then, we re-arrange \eqref{eq:residual_computation} for general $v\in X$ to find also that $\int_{I_n}\norm{f-f_\tau}_{H^{-1}(\Omega)}^2\mathrm{d}t$ is also bounded by the right-hand side of~\eqref{eq:semidisc_lower_dataosc_proof}.
\end{proof}

\begin{remark}[Bibliographical remarks]
A lower bound for the jump estimator similar to those in~\eqref{eq:semidisc_lower_bound_jump} and~\eqref{eq:time_local_lower_bound_jump} was first shown by Verf\"urth in~\cite{Verfurth2003}, in the case of a fully discrete approximation using the $\theta$ method in time, which includes the implicit Euler method as a special case.
The approach adopted in~\cite{Verfurth2003} towards proving the bound is substantially different to the one shown here, leading to a constant that possibly depends on the shape-regularity of the spatial meshes. 
The analysis for the temporal jump estimator was later improved by Ern, Smears \& Vohral\'{i}k in \cite{ESV2017}, which treats more general discontinuous Galerkin (DG) time-discretizations of arbitrary order, and which includes the implicit Euler method as a special case.
In particular, the lower bound for the jump estimator given in~\cite[Theorem~5.1]{ESV2017} features an explicit efficiency constant that is independent of the spatial discretization and that is robust with respect to the temporal polynomial degree.
The lower bound of Theorem~\ref{thm:semidisc_data_osc_sharp} that removes the data oscillation from the right-hand sides appears to be original.
Theorem~\ref{thm:semidisc_data_osc_sharp} showcases an example where the data oscillation term is dominated by the error, in a similar spirit to the work of Kreuzer \& Veeser in~\cite{KreuzerVeeser2021} for elliptic problems.
\end{remark}

\subsection{A posteriori error bound for the $L^2(H^1)$-norm}\label{sec:X-norm-bounds}

We now turn to the bounding the $L^2(H^1)$-norms of the error, namely $\norm{u-u_\tau}_X$ and $\norm{u-U_\tau}_X$.
We shall consider both of these quantities, since, depending on the situation, either $u_\tau$ or $U_\tau$ might be the better approximation of $u$ in the $X$-norm, see Example~\ref{ex:comparison}.
Note that the previous section on the \emph{a posteriori} error bounds for the $Y$-norm, where it is recalled that $Y=L^2(0,T;H^1_0(\Omega))\cap H^1(0,T;H^{-1}(\Omega))$ immediately yield, at least in some sense, upper bounds on the error in the $X$-norm, with $X=L^2(0,T;H^1_0(\Omega))$. 
However, as will be seen below, a more direct approach yields sharper results, especially in terms of the data oscillation.
Recall that the temporal jump estimator~$\eta_J$ is defined in~\eqref{eq:jump_estimator_semidiscrete} above.
 
\begin{theorem}\label{thm:semidiscrete_X_bound}
Suppose that $u_0\in H^1_0(\Omega)$ and that $f_\tau$ is given by~\eqref{eq:time_mean_value}.
Then
\begin{subequations}\label{eq:semidiscrete_X_bound}
\begin{align}
\norm{u-u_\tau}_X \leq \eta_J + \etaOscX,\label{eq:semidiscrete_X_bound_ut} 
\\ \norm{u-U_\tau}_X \leq \eta_J + \etaOscX,
\label{eq:semidiscrete_X_bound_Ut}
\end{align}
\end{subequations}
where the data oscillation~$\etaOscX$ is defined by
\begin{equation}
\etaOscX \coloneqq \sup_{\vphi \in \YT\setminus\{0\}}\frac{\int_0^T\pair{f-f_\tau}{\vphi}\dd t}{\normYT{\vphi}}.
\end{equation}
\end{theorem}
\begin{proof}
We detail the proof of~\eqref{eq:semidiscrete_X_bound_ut}, and we leave the proof of~\eqref{eq:semidiscrete_X_bound_Ut} as an exercise to the reader.
We define the residual functional~$\calR_X \colon X\tends [\YT]^*$ by
\begin{equation}\label{eq:R_X_def}
\begin{aligned}
 \pair{\calR_X(v)}{\vphi}_{[\YT]^*\times \YT} \coloneqq \Bx(u-v,\vphi) &&&\forall v \in X,\; \forall \vphi \in \YT.
 \end{aligned}
 \end{equation}
Theorem~\ref{thm:inf_sup_X} implies that
\begin{equation}\label{eq:X_error_residual_equivalence}
\begin{aligned}
\norm{u-v}_{X} = \norm{\calR_X(v)}_{[\YT]^*} & & &\forall v\in X.
\end{aligned}
\end{equation}
Considering the residual $\pair{R_X(u_\tau)}{\varphi}_{Y_T^*\times Y_T}$ for arbitrary $\varphi \in Y_T$, we find that
\begin{multline}
\pair{R_X(u_\tau) }{\varphi}_{Y_T^*\times Y_T} = \int_0^T \left[ \pair{f}{\varphi} + \pair{\p_t \varphi}{u_\tau}  - (\nabla u_\tau,\nabla \varphi)_\Omega \right] \mathrm{d}t + (u_0,\varphi(0))_\Omega 
\\  = \int_0^T \underbrace{\left[\pair{f_\tau}{\varphi} - \pair{\p_t U_\tau}{\varphi}  - (\nabla u_\tau,\nabla \varphi)_\Omega \right]}_{=0 \text{ by } \eqref{eq:IE_1}} \mathrm{d}t 
 \\ +\int_0^T \left[\pair{\p_t \varphi}{u_\tau - U_\tau}  +\pair{f-f_\tau}{\varphi}\right] \mathrm{d}t,
\end{multline}
where we have used the identity based on integration-by-parts in time:
\begin{equation}\label{eq:semidiscrete_X_bound_1}
(u_0,\varphi(0))_\Omega = - \int_0^T \pair{\p_t \varphi}{U_\tau}  \mathrm{d}t - \int_0^T \pair{\p_t U_\tau}{\varphi} \mathrm{d}t \quad \forall \varphi \in Y_T.
\end{equation}
We now use the Cauchy--Schwarz inequality to bound
\begin{equation}\label{eq:semidiscrete_X_bound_2}
 \int_0^T \lvert\pair{\p_t \varphi}{u_\tau - U_\tau} \rvert \mathrm{d}t  \leq \norm{\varphi}_{Y_T} \norm{u_\tau - U_\tau }_X.
\end{equation}
where it is recalled that~$\norm{\cdot}_{Y_T}$ is defined in~\eqref{eq:YT_def} above.
Therefore, the equivalence between error and residual~\eqref{eq:X_error_residual_equivalence} implies that
\begin{equation}
\norm{u-u_\tau}_X = \sup_{\varphi \in Y_T\setminus\{0\}}\frac{\pair{R_X(u_\tau)}{\varphi}_{Y_T^*\times Y_T} }{\norm{\varphi}_{Y_T}} \leq \eta_J + \etaOscX .
\end{equation}
This completes the proof of \eqref{eq:semidiscrete_X_bound_ut}.
\end{proof}

\begin{remark}[Comparison of data oscillation terms]
To see the difference between the bounds of Theorems~\ref{thm:apost_semidisc_Y_norm} and \ref{thm:semidiscrete_X_bound}, consider for example the case where $f\in L^2(0,T;L^2(\Omega))$ and where $f_\tau$ is the temporal mean-value projection of $f$, i.e. $f_\tau$ given by~\eqref{eq:time_mean_value}.
Then, it is known from~\cite{ESV2019} that one obtains the following bound for the data oscillation $\etaOscX$:
\begin{equation}
\etaOscX \leq \left( \sum_{n=1}^N \frac{\tau_n}{2\pi} \norm{f-f_\tau}_{L^2(I_n;L^2(\Omega))}^2 \right)^{\frac{1}{2}}.
\end{equation}
Thus we see that, in many cases, $\etaOscX$ can converge to zero faster by an additional half-order in the time-step size in comparison to $\etaOscY$.
Note also that $\etaOscX$ is often higher-order compared to the error measured in the $X$-norm.
By contrast, $\etaOscY$ can be of the same order as the error in the $Y$-norm, although it is bounded by the error in the case of mean-value data approximations~\eqref{eq:time_mean_value}, as shown by Theorem~\ref{thm:semidisc_data_osc_sharp}.
\end{remark}

\paragraph{The problem of lower bounds for the $X$-norm errors.}
Theorem~\ref{thm:semidiscrete_X_bound} gives upper bounds on the errors $\norm{u-u_\tau}_X$ and $\norm{u-U_\tau}_X$.
However, no lower bound is given.
We address this matter below, where we shall see that the jump estimator is not generally efficient with respect to the smallest between $\norm{u-u_\tau}_X$ and $\norm{u-U_\tau}_X$, but it is efficient with respect to the largest of these quantities.

\begin{example}\label{ex:inefficiency_jump_X_norm}
Consider the setting of Example~\ref{ex:comparison} with the same notation. Note that this corresponds to a problem where $f$ is constant and thus there is no data oscillation.
Then, it follows from~\eqref{eq:ex_asymptot_1} that
\begin{equation}
 \lim_{\lambda\tends \infty} \frac{\norm{u-u_\tau}_{X}}{\eta_J} = \lim_{\lambda\tends \infty} \frac{\norm{u-u_\tau}_{L^2(0,1)}}{\norm{u_\tau-U_\tau}_{L^2(0,1)}}  = 0.
\end{equation}
In this case, the estimator $\eta_J$ is not efficient relative to the error $\norm{u-u_\tau}_X$ when $\lambda$ becomes large.
It also follows from~\eqref{eq:ex_asymptot_2} that
\begin{equation}
 \lim_{\lambda\tends 0} \frac{\norm{u-U_\tau}_{X}}{\eta_J} = \lim_{\lambda\tends 0} \frac{\norm{u-U_\tau}_{L^2(0,1)}}{\norm{u_\tau-U_\tau}_{L^2(0,1)}} = 0.
\end{equation}
In this case, the estimator $\eta_J$ is not efficient relative to the error $\norm{u-U_\tau}_X$ when $\lambda$ becomes small.
\end{example}

Example~\ref{ex:inefficiency_jump_X_norm} shows that, in general, either of the effectivity indices $\frac{\eta_J}{\norm{u-u_\tau}_X}$ and $\frac{\eta_J}{\norm{u-U_\tau}_X}$ can be arbitrarily large. In other words, there are cases where the jump estimator can significantly overestimate the smallest $X$-norm error. It is therefore not possible to show that the estimator $\eta_J$ is bounded by either $\norm{u-u_\tau}_X$ or $\norm{u-U_\tau}_X$.
At present, it is not known how to obtain a computable estimator that is efficient with respect to the smallest $X$-norm error. 

\begin{remark}[Alternative error measures]\label{rem:error_measure_X}
The lack of efficiency of the temporal jump estimators in general for the $X$-norm error has motivated various works in the literature to consider combinations of $\norm{u-u_\tau}_X$ and $\norm{u-U_\tau}_X$, along with other quantities, as some stronger measure of the error, see for instance~\cite{AkrivisMakridakisNochetto2006,AkrivisMakridakisNochetto2009,MakridakisNochetto2006,NochettoSavareVerdi2000,Schotzau2010}.
For example, one can consider the error measure
\begin{equation}\label{eq:extend_X_norm_measure}
\mathcal{E}_X\coloneqq \norm{u-u_\tau}_X + \norm{u-U_\tau}_X.
\end{equation} 
The motivation for considering these augmented error measures is that the jump estimator is then trivially a lower bound for $\mathcal{E}_X$, as a result of the triangle inequality $\eta_J \leq \mathcal{E}_X$.
However, it is clear from Example~\ref{ex:inefficiency_jump_X_norm} that there is no general equivalence between $\mathcal{E}_X$ (or similar augmented error measures) with either of $\norm{u-u_\tau}_X$ or $\norm{u-U_\tau}_X$ considered individually.
See Remark~\ref{rem:relations} below on how $\mathcal{E}_X$ is, in some sense, more naturally related to the \emph{energy norm} error for an alternative reconstruction of the numerical solution.
\end{remark}

\begin{remark}[Rate of convergence of the estimator for smooth solutions]
Although $\eta_J$ is not generally efficient with respect to the smallest of $\norm{u-u_\tau}_X$ and $\norm{u-U_\tau}_X$, this does not contradict the fact that $\eta_J$ still has optimal rates of convergence with respect to the time-step sizes if the solution has some additional regularity. 
As mentioned already in Section~\ref{sec:IE_Discretization}, for a fixed problem with $u$ sufficiently regular, then \emph{a priori} error analysis shows that 
$\norm{u-u_\tau}_X$, $\norm{u-U_\tau}_X$ and thus also $\eta_J$ all have an overall first-order rate of convergence with respect to the time-step sizes.
We stress that there is no contradiction here with the earlier discussion of the lack of efficiency of the estimators, since the \emph{a priori} analysis results concern a fixed problem and considers the limit of small time-steps, whereas the efficiency of the estimators is a stronger property which must take into account a large range of problems and choices of time-step size.
\end{remark}

\subsection{A posteriori error bound for the energy norm error and a hypercircle theorem}\label{sec:hypercircle_theorem}

We now turn to the \emph{a posteriori} error bounds for the energy norm.
Recall that by definition $u_\tau$ is taken to be left-continuous and satisfies $u_\tau(t_n)=u_n$, for all $n=0,\dots,N$, c.f.\ \eqref{eq:IE_as_FD}.
Therefore, it makes sense to consider the energy norm of the error between $u$ and its piecewise constant approximation $u_\tau$, namely
\begin{equation}
\norm{u-u_\tau}_E^2 = \frac{1}{2}\norm{u(T)-u_\tau(T)}_\Omega^2+\int_0^T\norm{\nabla(u-u_\tau)}_\Omega^2\dd t.
\end{equation}
Likewise, we can also consider the energy norm of the error between $u$ and the continuous piecewise affine approximation $U_\tau$, given by
\begin{equation}
\norm{u-U_\tau}_E^2 = \frac{1}{2}\norm{u(T)-U_\tau(T)}_\Omega^2 + \int_0^T\norm{\nabla(u-U_\tau)}_\Omega^2\dd t.
\end{equation}
We will see shortly below (c.f.\ Corollary~\ref{cor:hypercircle}) that, for the special case of vanishing data oscillation, i.e. $f=f_\tau$, we have the Pythagoras identity
\begin{equation}\label{eq:hypercircle_energy}
\norm{u-u_\tau}_E^2 + \norm{u-U_{\tau}}_E^2 = \norm{u_\tau-U_\tau}_E^2 = 
\eta_J^2,
\end{equation}
where we recall that the temporal jump estimator~$\eta_J$ is defined in~\eqref{eq:jump_estimator_semidiscrete} above.
Observe that the equality $\norm{u_\tau-U_\tau}_E = \norm{u_\tau-U_\tau}_X = \eta_J $ holds since $u_\tau(T)=U_\tau(T)=u_N$ the value of $u_\tau$ on the final time-interval $I_N$.
Similar to the problems encountered with regards to the efficiency of $\eta_J$ in the $X$-norm setting, we note that Example~\ref{ex:inefficiency_jump_X_norm} above can be extended to show that, in general, the estimator $\eta_J$ is not an efficient one for either of $\norm{u-u_\tau}_E$ or $\norm{u-U_\tau}_E$.

However, the Pythagoras identity expressed by~\eqref{eq:hypercircle_energy} suggests that some notion of efficiency can be recovered by changing point of view, since it is equivalent to
\begin{equation}\label{eq:hypercircle_energy_2}
\begin{aligned}
\norm{u-\uha }_E = \frac{1}{2}\norm{u_\tau-U_\tau}_E = \frac{1}{2}\eta_J, &&& \uha \coloneqq \frac{1}{2}(u_\tau+U_\tau).
\end{aligned}
\end{equation}
The identity~\eqref{eq:hypercircle_energy_2} shows that the exact solution~$u$ and its approximations $u_\tau$ and $U_\tau$ all lie on a hypercircle of radius $\frac{1}{2}\eta_J$ and centre $\uha$, when the geometry is determined by the inner-product related to the energy norm. 
Thus, the estimator $\eta_J$ is efficient (in fact, exact) with regards to the energy norm error when we regard the function $\uha$ as the notion of numerical solution to be compared against the true solution.
We shall see that it is the second form~\eqref{eq:hypercircle_energy_2} that generalizes most readily to the general case where we no longer require the data oscillation to vanish exactly.

\begin{theorem}\label{thm:energy_norm_bound}
Let $\uha \coloneqq \frac{1}{2}(u_\tau+U_\tau)$.
We have the bounds
\begin{subequations}\label{eq:energy_norm_bounds}
\begin{align}
\norm{u-\uha}_E &\leq \frac{1}{2}\eta_J + \etaOscE,
\\ \frac{1}{2}\eta_J &\leq \norm{u-\uha}_E + \etaOscE,
\end{align}
\end{subequations}
where the data oscillation term $\etaOscE$ is defined by
\begin{equation}
\etaOscE \coloneqq \sup_{\varphi\in Y\setminus\{0\}}\frac{\int_0^T \pair{f-f_\tau}{\varphi}\dd t}{\normYs{\varphi}}.
\end{equation}
\end{theorem}

\begin{proof}
Recall that $Z=X\times L^2(\Omega)$.
We start by defining $(\bm{u}_\tau,\bm{U}_\tau,\overline{\bm{u}}_\tau) \in Z^3$ by
\begin{equation}
\begin{aligned}
\bm{u}_\tau \coloneqq (u_\tau,u_\tau(T)), && \bm{U}_\tau \coloneqq (U_\tau,U_\tau(T)) && \overline{\bm{u}}_\tau \coloneqq (\uha,\uha(T)).
\end{aligned}
\end{equation}
We also define $\bm{u}\coloneqq (u,u(T))\in Z$. Then, Theorem~\ref{thm:Z_infsup} shows that
\begin{equation}\label{eq:energy_norm_bound_1}
\norm{u-\uha}_E = \norm{\bm{u}-\overline{\bm{u}}_\tau}_Z = \sup_{\varphi\in Y\setminus\{0\}}\frac{\Bz(\bm{u}-\overline{\bm{u}}_\tau,\varphi)}{\normYs{\varphi}}.
\end{equation}
It follows from the weak formulation~\eqref{eq:Z_weakform} that, for any $\varphi\in Y$,
\begin{multline}\label{eq:energy_norm_bound_2}
\Bz(\bm{u}-\overline{\bm{u}}_\tau,\varphi) = \int_0^T \pair{f}{\varphi}\mathrm{d}t + (u_0,\varphi(0))_{\Omega}
\\ - (\overline{u}_\tau(T),\varphi(T))_\Omega - \int_0^T\left[-\pair{\p_t \varphi}{\overline{u}_\tau}+(\nabla \overline{u}_\tau,\nabla \varphi)_{\Omega}\right]\mathrm{d}t.
\end{multline}
Consider now a fixed but arbitrary $\varphi \in Y$.
Integrating-by-parts with respect to time and using the identities $U_\tau(T)=u_\tau(T)=\overline{u}_\tau(T)$, $U_\tau(0)=u_0$, we obtain from~\eqref{eq:semidiscrete_identity_1} that
\begin{equation}\label{eq:energy_norm_bound_3}
    0= (\overline{u}_\tau(T),\varphi(T))_\Omega - (u_0,\varphi(0))_{\Omega}  + \int_{0}^T\left[-\pair{\p_t \varphi}{U_\tau} + (\nabla u_\tau,\nabla \varphi)_\Omega - \pair{f_\tau}{\varphi}\right]\mathrm{d}t.
\end{equation}
Therefore, adding~\eqref{eq:energy_norm_bound_3} to \eqref{eq:energy_norm_bound_2} above and noting some cancellations, we find that 
\begin{equation}\label{eq:energy_norm_bound_4}
\Bz(\bm{u}-\overline{\bm{u}}_\tau,\varphi) = \int_0^T \pair{f-f_{\tau}}{\varphi}\mathrm{d}t
 + \int_0^T \left(\pair{\p_t \varphi}{\overline{u}_\tau-U_\tau} + (\nabla(u_\tau-\overline{u}_\tau),\nabla \varphi)_\Omega \right)\mathrm{d}t.
\end{equation}
Crucially, observe that 
\begin{equation}
\overline{u}_\tau-U_\tau = u_\tau - \overline{u}_{\tau} = \frac{1}{2}(u_\tau-U_\tau).
\end{equation}
Hence, \eqref{eq:energy_norm_bound_4} simplifies to
\begin{equation}
\Bz(\bm{u}-\overline{\bm{u}}_\tau,\varphi) = \frac{1}{2}\int_0^T \left[ \pair{\p_t \varphi}{u_\tau-U_\tau}+(\nabla \varphi,\nabla (u_\tau-U_\tau))_\Omega\right] \dd t   + \int_0^T\pair{f-f_\tau}{\varphi}\dd t,
\end{equation}
which holds for arbitrary $\varphi \in Y$.
Therefore, to prove~\eqref{eq:energy_norm_bounds}, it is enough to show that
\begin{equation}\label{eq:energy_norm_bound_7}
\sup_{\varphi\in Y\setminus\{0\}} \frac{\int_0^T \left[ \pair{\p_t \varphi}{u_\tau-U_\tau}+(\nabla \varphi,\nabla (u_\tau-U_\tau))_\Omega\right] \dd t}{\normYs{\varphi}} = \eta_J.
\end{equation}
To show~\eqref{eq:energy_norm_bound_7}, let $R\colon Y\tends Y$ denote the time-reversal map $R\vphi = \vphi(T-\cdot)$. Note that $\normYs{\R \vphi} = \normYs{\vphi}$ for all $\vphi \in Y$, i.e.\ $R$ is an isometry when $Y$ is equipped with the norm $\normYs{\cdot}$. 
It is clear that a change of variables gives
\begin{equation}
\int_0^T \left[ \pair{\p_t \varphi}{u_\tau-U_\tau}+(\nabla \varphi,\nabla (u_\tau-U_\tau))_\Omega\right] \dd t = \Bz (\bm{v}, R\vphi ),
\end{equation}
where $\bm{v}\coloneqq (R (u_\tau-U_\tau), 0) \in Z$. Therefore, we obtain~\eqref{eq:energy_norm_bound_7} from the identities
\begin{multline}
\sup_{\vphi \in Y\setminus\{0\}}\frac{\Bz(\bm{v},R\vphi)}{\normYs{\vphi}}
  = \sup_{\vphi \in Y\setminus\{0\}}\frac{\Bz(\bm{v},R\vphi)}{\normYs{R\vphi}} 
   =  \sup_{\widetilde{\vphi} \in Y\setminus\{0\}}\frac{\Bz(\bm{v},\widetilde{\vphi})}{\normYs{\widetilde{\vphi}}} 
\\   \underset{\eqref{eq:Z_infsup_1}}{=} \norm{\bm{v}}_Z = \norm{R(u_\tau - U_\tau)}_X = \norm{u_\tau - U_\tau}_X = \eta_J.
\end{multline}
This completes the proof of~\eqref{eq:energy_norm_bounds}.
\end{proof}

Theorem~\ref{thm:energy_norm_bound} immediately implies the hypercircle theorem/Pythagoras identity as a corollary when the data oscillation vanishes.
\begin{corollary}[Hypercircle/Pythagoras identity]\label{cor:hypercircle}
Suppose that $f=f_\tau$.
Then~\eqref{eq:hypercircle_energy} and~\eqref{eq:hypercircle_energy_2} hold.
\end{corollary}

\begin{remark}[Relation to augmented error measures]\label{rem:relations}
Recalling Remark~\ref{rem:error_measure_X}, several works in the literature have treated augmented error measures, such as the quantity $\mathcal{E}_X$ defined in~\eqref{eq:extend_X_norm_measure} above.
The results in this section give some insight into these alternative error measures, since in the case of vanishing data oscillation, we have the equivalence
\begin{equation}
2\norm{u-\uha}_E\leq \mathcal{E}_X \leq 4 \norm{u-\uha}_E,
\end{equation}
which follows easily from $2\norm{u-\uha}=\eta_J \leq \mathcal{E}_X$ by the triangle inequality, and from $\mathcal{E}_X\leq 2 \eta_J=4\norm{u-\uha}_E $ by Theorem~\ref{thm:semidiscrete_X_bound}.
Therefore, the quantity $\mathcal{E}_X$ is seen to be equivalent to the energy norm error $\norm{u-\uha}_E$, at least globally in space and in time.
\end{remark}

\begin{remark}[Bibliographical remarks]
Prager and Synge~\cite{PragerSynge1947} showed a hypercircle identity for approximations of the system of elasticity.
In fact, Prager and Synge~\cite[p.~248]{PragerSynge1947} originally suggested the idea of considering the center of the circle as the approximate solution, in this case $\uha$ as we have done in Theorem~\ref{thm:energy_norm_bound} above.
Their result has since been extended to a wide range of problems, especially in the context of elliptic problems~\cite{BraessSchoberl2008,DestuynderMetivet1999,ESV2017b,ErnStephansenVohralik2010,SmearsVohralik2020,Vohralik2010}.
The hypercircle theorem is thus a fundamental concept in \emph{a posteriori} error analysis, and it is well-understood how it relates to the primal and dual formulations of problems that have an energy-minimization principle.
\end{remark}

\section{Discretization in time and in space}\label{sec:fully_discrete}

We now turn to a fully discrete approximation of~\eqref{eq:parabolic}. For simplicity, we shall restrict our attention to the simplest case of a conforming finite element method on a fixed spatial mesh coupled with the implicit Euler discretization in time.
In order to make use of some results in the literature, we shall restrict our attention to problems in at most three spatial dimensions, i.e.\ we now assume that $1\leq \dim \leq 3$.
We further assume that the bounded open set $\Omega$ is in additional polytopal and has Lipschitz boundary $\partial \Omega$.
Let $\T$ be a conforming simplicial mesh on $\Omega$.
For each element $K\in\T$, we let $h_K$ denote the diameter of $K$. We define the mesh-size function $h_{\T}\in L^\infty(\Omega)$ by $h_{\T}|_K\coloneqq h_K$ for each element $K\in\T$.
In the following, the notation for inequalities $a \lesssim b$ will allow the hidden constant to depend on the shape-regularity parameter of $\T$, defined as $\theta_\T\coloneqq \max_{K\in\T}\frac{h_K}{\rho_K}$, where $\rho_K$ denotes the diameter of the largest inscribed balls in $K$. However, the hidden constants will be otherwise independent of the size of the mesh elements.
Let $p\geq 1$ denote a fixed integer which will correspond to the polynomial degree of the finite element space to be defined below.
The $H^1_0(\Omega)$-conforming finite element space of degree $p$ is denoted by $\VTp$, and is defined by
\begin{equation}\label{eq:VTp_def}
\VTp \coloneqq \{ v \in H^1_0(\Omega), \; v|_K \in \calP_{p}(K) \quad\forall K\in\T \},
\end{equation}
where $\calP_p(K)$ denotes the space of real-valued polynomials of total degree at most $p$ on $K$.
Keeping the same notation as in Section~\ref{sec:IE_Discretization} regarding the partition of $(0,T)$ into time-steps, we consider a fully discrete approximation  defined as follows: for each $n\in \{1,\dots, N\}$, find $u_{h,\tau,n}\in \VTp$ such that
\begin{equation}\label{eq:IE_FEM}
\begin{aligned}
\left(\frac{u_{h,\tau,n}-u_{h,\tau,n-1}}{\tau_n}, v_{h}\right)_\Omega + (\nabla u_{h,\tau,n},\nabla v_h)_\Omega = (f_{h,\tau,n},v_{h})_\Omega  &&& \forall v_h \in \VTp,
\end{aligned}
\end{equation}
where $u_{h,\tau,0} \in \VTp$ is some discrete approximation of the initial datum $u_0$, and where $f_{h,\tau,n}\in L^2(\Omega)$ is some discrete approximation of $f(t_n)$. In the analysis below, we allow for rather general choices of $f_{h,\tau,n}$, with the only assumptions required being that $f_{h,\tau,n}|_K \in \mathcal{P}_{p}(K)$ for each $K\in \T$, for each $n\in\{1,\dots,N\}$.

\paragraph{Piecewise-constant-in-time reconstruction.}

Let $\Vhtp$ denote the space of all left-continuous functions $v\colon [0,T]\rightarrow \VTp$ that are piecewise constant with respect to the time partition $\{I_n\}_{n=1}^N$, i.e.\ $v|_{I_n}\in \mathcal{P}_0(I_n;\VTp) $ for each $n\in\{1,\dots,N\}$, where  we recall that $I_n=(t_{n-1},t_n)$ is the $n$-th time-interval.
Note that a function  $v\in \Vhtp$ thus has a well-defined value $v(t)\in \VTp$ for each time $t\in[0,T]$. Note that from this point of view, functions in $\Vhtp$ that agree for almost all $t\in [0,T]$ are not identified.
We then define $\uht\in \Vhtp$ as the unique function that satisfies
\begin{equation}
\begin{aligned}
\uht(0)=u_{h,\tau,0}, &&& \uht|_{I_n} =u_{h,\tau,n} \quad \forall n\in\{1,\dots,N\}.
\end{aligned}
\end{equation}
Observe also that left-continuity of $\uht$ ensures that $\uht(t_n) = u_{h,\tau,n}$ for each $n\in \{0,\dots, N\}$.

\paragraph{Continuous piecewise-affine-in-time reconstruction.}
We define also $\Uht \colon [0,T]\rightarrow \VTp$ as the unique continuous piecewise-affine function that satisfies
\begin{equation}\label{eq:fd_affine_recons}
\begin{aligned}
\Uht(t)= \frac{t-t_{n-1}}{\tau_n} u_{h,\tau,n} + \frac{t_n-t}{\tau_n} u_{h,\tau,n-1} &&& \forall t \in [t_{n-1},t_n].
\end{aligned}
\end{equation}
Observe that, by definition, $\Uht(t_n)=u_{h,\tau,n}$ for each $n\in \{0,\dots,N\}$.

\paragraph{Reformulation of the fully discrete scheme.}
We now define the function $f_{h,\tau} \in L^2(0,T;L^2(\Omega))$ by $f_{h,\tau}|_{I_n} \coloneqq f_{h,\tau,n}$ for $n\in \{1,\dots,N\}$, where $f_{h,\tau,n}$ is the chosen approximation of $f(t_n)$ appearing in \eqref{eq:IE_FEM} above.
Under the assumptions above on the $f_{h,\tau,n}$, it follows that $f_{h,\tau}$ is piecewise constant in time with respect to the time-intervals $\{(t_{n-1},t_n)\}_{n=1}^N$ and piecewise polynomial of degree at most $p$ in space with respect to the mesh $\T$.
It is then clear that \eqref{eq:IE_FEM} is equivalent to: 
\begin{equation}\label{eq:IE_FEM_2}
  ( \p_t \Uht (t) , v_h )_\Omega +  (\nabla \uht(t) , \nabla v_h)_\Omega = (f_{h,\tau}(t),v_h)_\Omega  \quad \forall v_h\in \VTp, \quad \text{a.e.\ } t\in I_n.
\end{equation}
for each $n\in \{1,\dots, N\}$.

\subsection{Construction of the equilibrated flux}

We consider now \emph{a posteriori} error bounds for the error based on the construction of an equilibrated flux. 
We construct a discrete and locally computable $H(\mathrm{div})$-conforming vector field $\sht$ that satisfies the equilibration property
\begin{equation}\label{eq:flux_equilibration}
\begin{aligned}
\p_t U_{h,\tau} + \nabla{\cdot}\sht = f_{h,\tau} &&& \text{in } \Omega\times(0,T),
\end{aligned}
\end{equation}
where $U_{h,\tau}$ is defined above, and $f_{h,\tau}$ will be further specified below.
The construction follows the approach from~\cite{ESV2017,ESV2019}, which provides a general construction that was designed to accommodate variable polynomial degrees in $hp$-FEM and also mesh modification between the time-steps. 
Since we are restricting ourselves here to the case of a single polynomial degree for all mesh elements and a single mesh for all time-steps, we shall give here a slightly simplified version of the construction from~\cite{ESV2017,ESV2019}.
The construction of the equilibrated flux is based on solving local mixed FEM problems using Raviart--Thomas--N\'ed\'elec (RTN) elements on the patches of elements surrounding each vertex of the mesh. The problem for each vertex-patch is independent from the others, so could be solved in parallel. 
For an introduction to mixed finite element methods, including RTN spaces, we refer the reader to the books~\cite{BoffiBrezziFortin2013,ErnGuermond2004}.
The final global equilibrated flux $\sht$ will then be obtained by summing the contributions from each vertex patch.

\paragraph{Local equilibration on vertex patches.}
Let $\Ver$ denote the set of vertices of the mesh~$\T$.
Let $\Ver_{\Omega}=\Ver\cap \Omega$ denote the set of interior vertices and let $\Ver_{\partial\Omega}=\Ver\cap \partial\Omega$ denote the boundary vertices.
For each vertex $\ver\in\Ver $, let $\psia$ denote the hat function associated with $\ver$. 
Note that $\{\psia\}_{\ver\in\Ver}$ forms a partition of unity of $\overline{\Omega}$, i.e.\ $\sum_{\ver\in\Ver}\psia(x)=1$ for all $x\in \overline{\Omega}$. 
Let $\Ta$ denote the set of all elements of $\T$ that contain $\ver$, and let $\oma\coloneqq \bigcup \{ K \colon K\in\Ta \}$ denote the vertex patch around $\ver$, i.e.\ the union of all elements of $\Ta$.
With the convention that elements of $\T$ are considered to be closed sets, it follows that $\oma$ is the support of $\psia$.
Let the integer $\widetilde{p}\geq p+1$ denote a fixed choice of polynomial degree that will be used for the flux reconstruction, where it is recalled that $p$ is the polynomial degree used for $\VTp$ above. 
For each $\ver\in \Ver$, let the spaces $\calP_{\widetilde{p}}(\Ta)$ and $\RTNa$ be defined by
\begin{subequations} 
\begin{align}
\calP_{\widetilde{p}}(\Ta) &\coloneqq \{ q_h \in L^2(\oma)\colon q_h|_K \in \calP_{\widetilde{p}}(K)\quad\forall K\in\Ta\},\\
\RTNa &\coloneqq \{ \bm{v}_h \in L^2(\oma;\R^\dim) \colon \bm{v}_h|_K \in \RTNK \quad\forall K\in \Ta \},
\end{align}
\end{subequations}
where $\RTNK\coloneqq \calP_{\widetilde{p}}(K;\R^\dim)+\bm{x}\calP_{\widetilde{p}}(K)$ denotes the RTN space of order $\widetilde{p}$ on $K$. In other words, $\calP_{\widetilde{p}}(\Ta)$ denotes the space of scalar functions that are piecewise polynomials of degree at most $\widetilde{p}$ with respect to $\Ta$, and $\RTNa$ denotes the space of vector fields that are piecewise RTN of order $\widetilde{p}$ with respect to $\Ta$. Note that there are no continuity conditions across mesh elements in the definition of these spaces.
We now define the local mixed finite element spaces $\Qa$ and $\Va$ by
\begin{subequations}
\begin{align}
\Qa &\coloneqq \begin{cases}
\left\{q_h \in \calP_{\widetilde{p}}(\Ta), \quad \int_\oma q_h \mathrm{d}x = 0 \right\} &\text{if }\ver\in\Ver_{\Omega} ,
\\
\calP_{\widetilde{p}}(\Ta) &\text{if }\ver\in\Ver_{\partial\Omega},
\end{cases}\label{eq:Qa_def}
\\
\Va &\coloneqq 
\begin{cases}
\left\{\bm{v}_h \in H(\Div,\oma) \cap \RTNa\, \quad \bm{v}_h\cdot \bm{n} =0 \text{ on } \partial \oma \right\} &\text{if } \ver\in \Ver_{\Omega},
\\
\left\{\bm{v}_h \in H(\Div,\oma) \cap \RTNa\, \quad \bm{v}_h\cdot \bm{n} =0 \text{ on } \partial \oma \setminus \partial \Omega \right\} &\text{if } \ver\in \Ver_{\partial\Omega},
\end{cases}
\end{align}
\end{subequations}
where $\bm{n}$ denotes the unit outward normal on $\partial \oma$.
Thus, in the case of an interior vertex $\ver\in\Ver_{\Omega}$, the space $\Va$ is the $H(\Div,\oma)$-conforming RTN space of order $\widetilde{p}$ with a constraint of vanishing normal component on the boundary $\partial \oma$. 
For a boundary vertex $\ver\in\Ver_{\partial\Omega}$, the normal component of functions in $\Va$ are constrained only on the faces of the patch-boundary $\partial\oma$ that are not part of the domain boundary $\partial\Omega$. In other words, for interior vertices, these spaces correspond to the usual RTN mixed finite element spaces for an elliptic problem over $\oma$ with a homogeneous Neumann boundary condition on $\partial \oma$, whereas for boundary vertices, these correspond to RTN mixed finite element spaces for a problem with mixed Neumann and Dirichlet conditions, with a homogeneous Neumann condition only on $\partial\oma\setminus\partial\Omega$. 
For each $\ver\in \Ver$ and each $n\in {1,\dots,N}$, define the piecewise polynoimal function $g_{h,\tau}^{\ver,n}\colon \oma \rightarrow \R$ by
\begin{equation}
g_{h,\tau}^{\ver,n} \coloneqq  \psia f_{h,\tau}|_{\oma\times I_n} - \psia \p_t \Uht|_{\oma\times I_n} - \nabla\psia{\cdot}\nabla \uht|_{\oma\times I_n}.
\end{equation}
Note in particular that the choice $\widetilde{p}\geq p+1$ is used above to ensure that the terms $\psia f_{h,\tau,n}$ and $\psia \p_t \Uht$ are piecewise polynomials of degree less than or equal to $\widetilde{p}$ with respect to the spatial mesh $\T$. Therefore $g_{h,\tau}^{a,n} \in \calP_{\widetilde{p}}(\Ta)$ for each $\ver\in\Ver$.
Note also that in the case of an interior vertex $\ver \in \Ver_{\partial\Omega}$, choosing as a test function $v_h=\psia$ in~\eqref{eq:IE_FEM_2} above implies that
\begin{equation}\label{eq:source_compatibility_condition}
\int_{\oma} g_{h,\tau}^{\ver,n} \dd x = (f_{h,\tau,n},\psia)_\Omega - (\p_t\Uht,\psia)_\Omega - (\nabla \uht,\nabla \psia)_\Omega =0,
\end{equation}
where we recall that $f_{h,\tau}|_{I_n}=f_{h,\tau,n}$ by definition.
Therefore, $g_{h,\tau}^{\ver,n}$ has zero-mean value in the case of an interior vertex $\ver$, and thus we conclude that that $g_{h,\tau}^{\ver,n}\in \Qa$ for all vertices $\ver\in \Ver$, where it is recalled that $\Qa$ was defined in~\eqref{eq:Qa_def} above.
Next, for each time-step $n\in\{1,\dots,N\}$ and each vertex $\ver\in \Ver$, we define the local flux contribution $\sh^{\ver,n} \in \Va$ by 
 \begin{equation}
\sht^{\ver,n}\coloneqq \argmin_{\substack{\bm{v}_{h,\tau}\in \Va \\ \nabla\cdot\bm{v}_{h,\tau} = g_{h,\tau}^{\ver,n}}} \norm{\bm{v}_{h,\tau}+\psia\nabla u_{h,\tau}|_{I_n}}_{\oma}.
\end{equation}
Crucially, the fact that $\sh^{\ver,n}$ is well-defined in the case of an interior vertex is a consequence of~\eqref{eq:source_compatibility_condition}.
Indeed, this corresponds to the zero mean-value compatibility condition that is required on the source term when considering a mixed finite element discretization of a Poisson equation with pure Neumann boundary conditions. The case of boundary vertices is simpler still, because it is then not necessary to require that $\int_{\oma} g_{h,\tau}^{\ver,n} \mathrm{d}x=0$ since vector fields in $\Va$ are then only constrained on part of the boundary of $\oma$.

In order to pass from the local flux contributions to a global equilibrated flux, we now extend the individual local contributions by zero to the whole domain $\Omega$. Note that boundary conditions imposed on vector fields in $\Va$ imply that the extension by zero of $\sht^{\ver,n}$ to the whole domain $\Omega$, also denoted $\sh^{\ver,n}$, is in $H(\Div,\Omega)$.
Therefore, we may define the global flux $\sht\in L^2(0,T;H(\Div,\Omega))$ as the unique piecewise constant-in-time vector field such that
\begin{equation}\label{eq:sht_def}
\sht|_{I_n} \coloneqq \sum_{\ver\in\Ver} \sht^{\ver,n} \quad \forall n\in \{1,\dots,N\}.
\end{equation}
It follows that $\sht|_{I_n} \in H(\Div,\Omega)$ for each $n$, and thus $\sht\in L^2(0,T;H(\Div,\Omega))$. We find that the divergence $\nabla\cdot\sht$ satisfies then: 
\begin{equation}
\begin{split}
\nabla\cdot\sht|_{I_n}
& = \sum_{\ver\in\Ver} \nabla\cdot\sht^{\ver,n}  =\sum_{\ver\in\Ver} g_{h,\tau}^{\ver,n}
\\ & = \sum_{\ver\in \Ver} \left[\psia f_{h,\tau}|_{I_n} - \psia \p_t \Uht|_{I_n} - \nabla \psia \cdot \nabla \uht|_{I_n} \right]
\\ &= 
\sum_{\ver\in\Ver}\psia(f_{h,\tau}-\p_t\Uht)|_{I_n} - \left(\sum_{\ver\in\Ver}\nabla\psia\right)\cdot \nabla \uht|_{I_n}
\\ & = f_{h,\tau}|_{I_n}- \p_t \Uht|_{I_n} 
\end{split}
\end{equation}
where in passing to the last line above we have used the partition of unity property which entails that $\sum_{\ver\in \Ver}\psia \equiv 1$ in $\Omega$ and $\sum_{\ver\in\Ver}\nabla \psia \equiv 0$ in $\Omega$.
Thus, on re-arranging, we have the equilibration identity~\eqref{eq:flux_equilibration} above.

\begin{remark}[Data approximation]
For the sake of simplicity, we have considered here some abstract approximation $f_{h,\tau}$ that approximates $f$ without entering into too much detail. A more precise treatment can be found in~\cite{ESV2017}, which handles variable polynomials degrees between mesh elements, and gives a more specific construction of a suitable approximation of the right-hand side.
\end{remark}
  
\subsection{Error bound for the $L^2(H^1)\cap H^1(H^{-1})$ norm error}

Given the equilibrated flux~$\sht$ defined in~\eqref{eq:sht_def}, we now consider a posteriori error estimators for the $L^2(H^1)\cap H^1(H^{-1})$ norm of the error.

\begin{theorem}[Upper bound]\label{thm:upper_bound_Y_norm}
Let $\sht \in L^2(0,T;H(\Div,\Omega))$ be defined by~\eqref{eq:sht_def}. Then,
\begin{equation}\label{eq:Y_norm_upper_bound}
\norm{u-\Uht}_Y \leq \left(\int_0^T\left(\norm{\sht + \nabla \Uht}_{\Omega}+\norm{f-f_{h,\tau}}_{H^{-1}(\Omega)}\right)^2\mathrm{d}t + \norm{u_0-u_{h,\tau,0}}_\Omega^2 \right)^{\frac{1}{2}}.
\end{equation}
\end{theorem}

\begin{proof}
Theorem~\ref{thm:inf_sup_parabolic_Y} shows that
\begin{equation}\label{eq:Y_norm_upper_bound_1}
\norm{u-\Uht}_Y^2 = \norm{\mathcal{R}_Y(\Uht)}_{X^*}^2 + \norm{u_0-u_{h,\tau,0}}_\Omega^2,
\end{equation}
where $\mathcal{R}(\Uht)\in X^*$ was defined in~\eqref{eq:R_Y_def}, and where we note that $u(0)=u_0$ and $\Uht(0)=u_{h,\tau,0}$.
Using the equilibration identity~\eqref{eq:flux_equilibration}, we find that
\begin{equation}
\begin{split}
\langle \mathcal{R}_Y(\Uht),v\rangle & = \int_0^T \left[\pair{f}{v} - (\p_t\Uht,v)_\Omega - (\nabla\Uht,\nabla v)_\Omega\right]\mathrm{d}t
\\ &= \int_0^T \left[\pair{f-f_{h,\tau}}{v} + (f_{h,\tau} - \p_t \Uht,v)_\Omega - (\nabla\Uht,\nabla v)_\Omega\right]\mathrm{d}t
\\ &= \int_0^T \left[\pair{f-f_{h,\tau}}{v} + (\nabla\cdot \sht,v)_\Omega - (\nabla\Uht,\nabla v)_\Omega\right]\mathrm{d}t
\\ & = \int_0^T \left[\pair{f-f_{h,\tau}}{v}  - (\sht+\nabla\Uht,\nabla v)_\Omega\right]\mathrm{d}t,
\end{split}
\end{equation}
for all $v\in X$, where in passing to the last equality above we have used the identity $(\nabla\cdot\sht,v)_\Omega = - (\sht,\nabla v)_\Omega$ for all $v\in X$, a.e. in $(0,T)$.
Therefore, the Cauchy--Schwarz inequality and the definition of $\norm{\cdot}_{H^{-1}(\Omega)}$ imply that
\[
\abs{\langle \mathcal{R}_Y(\Uht),v\rangle}  \leq \int_0^T \left(\norm{f-f_{h,\tau}}_{H^{-1}(\Omega)}+\norm{\sht + \nabla \Uht}_\Omega\right) \norm{\nabla v}_\Omega \mathrm{d}t \quad \forall v\in X,
\]
and thus the Cauchy--Schwarz inequality implies that
\begin{equation}\label{eq:Y_norm_upper_bound_2}
\norm{\mathcal{R}_Y(\Uht)}_{X^*}^2 \leq \int_0^T \left(\norm{f-f_{h,\tau}}_{H^{-1}(\Omega)}+\norm{\sht + \nabla \Uht}_\Omega\right)^2 \mathrm{d}t.
\end{equation}
Combining \eqref{eq:Y_norm_upper_bound_1} and~\eqref{eq:Y_norm_upper_bound_2} then gives~\eqref{eq:Y_norm_upper_bound}.
\end{proof}

\begin{remark}[Comparison to estimators for semi-discrete approximations]
Notice that in the limit $h\tends 0$ where the spatial mesh is refined but the time-step $\tau$ is held fixed, we can expect that $\sht$ should approach $-\nabla \uht$. This can be justified at least heuristically by comparing \eqref{eq:semidiscrete_apost_identity} in the semi-discrete setting with the flux equilibration identity~\eqref{eq:flux_equilibration}.
Therefore, up to data oscillation, we can expect heuristically that the main error estimator term $\norm{\sht+\nabla\Uht}_{L^2((0,T)\times \Omega)}$ should approach the jump estimator $\eta_J$ from~\eqref{eq:jump_estimator_semidiscrete} in the limit of fine spatial meshes. 
\end{remark}

\begin{remark}[Bibliographical remarks]
Theorem~\ref{thm:upper_bound_Y_norm} presents a simplified form of the \emph{a posteriori} upper bound of~\cite[Corollary~5.3]{ESV2017}, which handled the more general case of varying polynomial-degrees between mesh elements, and varying meshes between time-steps. The result there is also more specific in its treatment of the data oscillation terms. 
\end{remark}

It is shown in~\cite[Corollary~5.3]{ESV2017} that the flux estimator $\eta_F$ is also efficient with respect to the $Y$-norm error of $u-\Uht$, at least \emph{locally} in time but only \emph{globally} in space, i.e.\ there is a bound of the form
\begin{equation}
\int_{I_n} \norm{\sht+\nabla \Uht}_\Omega^2 \mathrm{d}t \lesssim \int_{I_n} \left(\norm{\p_t(u-\Uht)}_{H^{-1}(\Omega)}^2+\norm{\nabla(u-\Uht)}_{\Omega}^2\right)\mathrm{d}t + \text{oscillation},
\end{equation}
for each $n\in\{1,\dots,N\}$, where we refer the reader to~\cite[Eqn.~(5.16)]{ESV2017} for the details on the oscillation term. However, a fully local efficiency bound for the $Y$-norm error $\norm{u-\Uht}_Y$, both with respect to space and time, is not known.
The same issue occurs also for residual estimators when considering the same error, see~\cite{Verfurth2003}. Thus, this issue appears to be related to the effect of temporal discretization and choice of norm for the error (see below), rather than any specific choice of construction of the \emph{a posteriori} error estimators.

\subsection{Error bound for an extended $L^2(H^1)\cap H^1(H^{-1})$-norm of the error}
The lack of local-in-space efficiency results for \emph{a posteriori} error estimators (both of residual-type or equilibrated flux-type) for the error quantity $\norm{u-\Uht}_Y$ motivates a change of viewpoint.
In place of considering $\Uht$ as the primary numerical solution, we can consider instead $\uht \in \Vhtp$ as the central object of interest. 
However, $\uht$ is generally not an element of $Y$, and thus one cannot evaluate directly the $Y$-norm of $u-\uht$. 
This motivates the definition of an extension of the norm on $Y$ to the vector sum\footnote{Recall that the sum space $Y+\Vhtp$ consists of all functions of the form $\varphi+v_{h,\tau}$ for $\varphi\in Y$ and $v_{h,\tau}\in \Vhtp$, and $Y+\Vhtp$ is a subspace of $X$. The sum is however not a direct sum.} space $Y+\Vhtp$, i.e.\ we define below a norm $\norm{\cdot}_{\mathcal{E}_Y}\colon Y+\Vhtp \tends \mathbb{R}_{\geq 0}$ such that $\norm{\varphi}_{\mathcal{E}_Y}=\norm{\varphi}_Y$ for all $\varphi\in Y$.
Recall that $Y$ is continuously embedded in $C([0,T];L^2(\Omega))$, and thus point values for all times $t\in[0,T]$ of functions in $Y$ can be understood in the trace sense.
We also recall that functions in $\Vhtp$ are left-continuous by definition.

Let $\mathcal{I}_{h,\tau}\colon Y+\Vhtp \tends Y $ be the linear operator defined by
\begin{equation}\label{eq:general_definition}
\mathcal{I}_{h,\tau} v(t)|_{I_n} \coloneqq v(t) + \frac{t_n-t}{\tau_n}\left(v(t_{n-1}^-)-v(t_{n-1}^+)\right) \qquad\forall n\in \{1,\dots,N\},
\end{equation}
where $v(t_{n-1}^-)=\lim_{\epsilon \searrow 0}v(t_n-\epsilon)$ and $v(t_{n-1}^+)=\lim_{\epsilon\searrow 0}v(t_{n-1}+\epsilon)$ are one-sided traces of $v$ at $t_{n-1}$, where we adopt the convention $v(t_0^-)=v(0)$.
To verify that $\mathcal{I}_{h,\tau}$ indeed maps $Y+\Vhtp$ to $Y$, observe firstly that the continuous embedding $Y\hookrightarrow C([0,T];L^2(\Omega))$ implies that $\varphi(t_{n-1}^+)=\varphi(t_{n-1}^-)$ for all $n\in\{1,\dots,N\}$, and therefore $\mathcal{I}_{h,\tau} \varphi = \varphi$ for any $\varphi\in Y$. 
Secondly, for a function $v\in \Vhtp$ and $t\in (t_{n-1},t_n]$, we have $v(t)=v(t_{n-1}^+)=v(t_n)$ and $v(t_{n-1}^-)=v(t_{n-1})$ since $v$ is piecewise constant in time and left-continuous, therefore
\begin{equation}\label{eq:interpolation_operator}
\mathcal{I}_{h,\tau} v (t) = \frac{t-t_{n-1}}{\tau_n} v(t_n)+\frac{t_n-t}{\tau_n} v(t_{n-1}) \quad \forall t\in (t_{n-1},t_n].
\end{equation}
Moreover $\mathcal{I}_{h,\tau} v(0)=v(0)$. This shows that $\mathcal{I}_{h,\tau}v $ is continuous, piecewise affine, with values in $\VTp$ for each $t\in[0,T]$, and thus belongs to the space $H^1(0,T;H^1_0(\Omega))\hookrightarrow Y$. This demonstrates that $\mathcal{I}_{h,\tau}\colon Y+\Vhtp\rightarrow Y$.
In particular, it follows from~\eqref{eq:fd_affine_recons} and \eqref{eq:interpolation_operator} above that
\begin{equation}\label{eq:interpolant_reconst}
\Uht = \mathcal{I}_{h,\tau} \uht.
\end{equation}

We now define the norm $\norm{\cdot}_{\mathcal{E}_Y}\coloneqq Y+\Vhtp\rightarrow \mathbb{R}_{\geq 0}$ by
\begin{equation}
\norm{w}_{\mathcal{E}_Y}^2 \coloneqq \norm{\calI_{h,\tau} w}_Y^2 + \norm{w - \calI_{h,\tau} w}_X^2 \quad\forall w\in Y+\Vhtp.
\end{equation}
Observe that $\norm{\cdot}_{\mathcal{E}_Y}$ satisfies the triangle inequality since $\mathcal{I}_{h,\tau}$ is linear.
Furthermore, $\norm{w}_{\mathcal{E}_Y}=0$ implies $w=\mathcal{I}_{h,\tau} w$ since $\norm{w-\mathcal{I}_{h,\tau} w}_X=0$ and thus $w=0$ since $\norm{w}_Y=\norm{\mathcal{I}_{h,\tau}w}_Y=0$.
Therefore $\norm{\cdot}_{\mathcal{E}_Y}$ is indeed a norm on $Y+\Vhtp$.
Additionally, $\norm{\varphi}_{\mathcal{E}_Y}=\norm{\varphi}_Y$ for all $\varphi\in Y$ since then $\varphi=\mathcal{I}_{h,\tau} \varphi$ as shown above.

In particular, observe that for the error $u-\uht$, we have from~\eqref{eq:interpolant_reconst} above that
\begin{equation}\label{eq:EY_norm_solution}
\norm{u-\uht}_{\mathcal{E}_Y}^2 = \norm{u - \Uht}_Y^2 + \norm{\uht - \Uht}_{X}^2.
\end{equation}
Thus the norm $\norm{u-\uht}_{\mathcal{E}_Y}$ includes both the $Y$-norm error $\norm{u-\Uht}_Y$ considered above as well as the the term $\norm{\uht - \Uht}_{X}$ which measures the lack of conformity of $\uht\not\in Y$. The quantity $\norm{\uht - \Uht}_{X}$ is also closely related to the jump estimator $\eta_J$ appearing in the semi-discrete setting, c.f. Section~\ref{sec:Y-norm-bounds} above.
Crucially, however, it is shown in~\cite[Theorem~5.1]{ESV2017} that we have the equivalence
\begin{equation}\label{eq:extended_norm_equiv}
\norm{u-\Uht}_Y \leq \norm{u-\uht}_{\mathcal{E}_Y} \leq 3 \norm{u-\Uht}_Y,
\end{equation}
see in particular~\cite[Eqn.~(5.8)]{ESV2017} and note that the coarsening error term that appears there vanishes in the current context where we consider only the case of a fixed spatial mesh $\T$ over all time-steps. Thus the norm $\norm{u-\uht}_{\mathcal{E}_Y}$ defines a notion of error that is globally equivalent to $\norm{u-\Uht}_Y$, yet  may have different localization, especially in terms of localization across the spatial domain. This enables \emph{a posteriori} error estimators that are then locally efficient both with respect to time and to space.

\begin{theorem}\label{thm:EY_norm_bound}
Let $\sht\in L^2(0,T;H(\Div,\Omega))$ be defined by~\eqref{eq:sht_def}. Then, we have the global upper bound
\begin{multline}\label{eq:EY_norm_upper_bound}
\norm{u-\uht}_{\mathcal{E}_Y} \leq \left( \int_0^T\left(\norm{\sht+\nabla\Uht}_\Omega+\norm{f-f_{h,\tau}}_{H^{-1}(\Omega)}\right)^2 \mathrm{d}t \right. \\ + \left. \int_0^T\norm{\nabla(\uht-\Uht)}_{\Omega}^2\mathrm{d}t + \norm{u_0-u_{h,\tau,0}}_\Omega^2 \right)^{\frac{1}{2}}.
\end{multline}
Furthermore, if $1\leq \dim \leq 3$, then for each $n\in\{1,\dots,N\}$ and each element $K\in\T$, we have the local lower bound
\begin{equation}\label{eq:EY_local_lower_bound}
\int_{I_n}\left(\norm{\sht+\nabla \Uht}_K^2+\norm{\nabla(\uht-\Uht) }_{K}^2\right)\mathrm{d}t
\lesssim \sum_{\ver\in \VK}\left(\abs{u-\uht}_{\mathcal{E}_Y^{\ver,n}}^2+\left[\eta_{\mathrm{osc}}^{\ver,n}\right]^2\right),
\end{equation}
where $\VK$ denotes the set of vertices of the element $K$, and
\begin{subequations}
\begin{gather}
\abs{u-\uht}_{\mathcal{E}_Y^{\ver,n}}^2 \coloneqq \int_{I_n} \left(\norm{\p_t(u-\Uht)}_{H^{-1}(\oma)}^2+\norm{\nabla(u-\Uht)}_\oma^2+\norm{\nabla(\uht-\Uht)}_\oma^2\right)\mathrm{d}t,
\\
\left[\eta_{\mathrm{osc}}^{\ver,n}\right]^2 \coloneqq \int_{I_n} \norm{f-f_{h,\tau}}_{H^{-1}(\oma)}^2\mathrm{d}t.\label{eq:local_data_osc}
\end{gather}
\end{subequations}
Moreover, we have the global lower bound
\begin{equation}\label{eq:EY_global_lower_bound}
\int_0^T \left(\norm{\sht+\nabla \Uht}_\Omega^2+\norm{\nabla(\uht-\Uht) }_{\Omega}^2\right)\mathrm{d}t \leq \norm{u-\uht}_{\mathcal{E}_Y}^2 + \sum_{n=1}^N\sum_{\ver\in\Ver}\left[\eta_{\mathrm{osc}}^{\ver,n}\right]^2.
\end{equation}
The hidden constants in~\eqref{eq:EY_local_lower_bound} and \eqref{eq:EY_global_lower_bound} depend only on the dimension $\dim$ of $\Omega$ and on the shape-regularity parameter of $\T$.
\end{theorem}
\begin{proof}
Since the full proof essentially follows~\cite[Theorem~5.2]{ESV2017} with only very minor simplifications, we shall only outline the main steps here. 
The upper bound~\eqref{eq:EY_norm_upper_bound} is an immediate consequence of \eqref{eq:Y_norm_upper_bound} and \eqref{eq:EY_norm_solution} above. 
To show the local lower bound~\eqref{eq:EY_local_lower_bound}, we start by noting that for each element $K\in\T$ and each vertex $\ver\in \Ver$, $\stha$ vanishes on $K$ unless $\ver\in \Ver_K$ is a vertex of $K$. Therefore, we have $(\sth+\nabla \uht)|_{K\times I_n}= \sum_{\ver\in\Ver_K}\left(\stha+\psia \nabla\uht\right)|_{K\times I_n}$. 
Therefore, we have
\begin{multline}
\int_{I_n}\left(\norm{\sth + \nabla \Uht}^2_K+\norm{\nabla(\uht-\Uht)}_K^2\right)\mathrm{d}t
\\ \lesssim 
\int_{I_n}\left(\norm{\sth + \nabla \uht}^2_K+\norm{\nabla(\uht-\Uht)}_K^2\right)\mathrm{d}t
 \\  \lesssim \sum_{\ver\in\Ver_K}\int_{I_n}\norm{\stha +\psia \nabla \uht}_{K}^2\mathrm{d}t+\int_{I_n}\norm{\nabla(\uht-\Uht)}_K^2\mathrm{d}t.
\end{multline}
To bound the first term on the right-hand side of the inequality above, we use the following stability result for local flux equilibration, due to \cite{BraessPillweinSchoberl2009} for the case $\dim=2$ and \cite{ErnVohralik2020} for $\dim=3$, which ultimately rests on the analytical results in~\cite{CostabelMcIntosh2010}, which shows that
\begin{equation}
\norm{\stha + \psia \nabla \uht}_{\oma} \lesssim \sup_{v\in H^1_*(\oma)\setminus\{0\}} \frac{(g_{h,\tau}^{\ver,n},v)_\oma - (\psia \nabla \uht,\nabla v)_\oma}{\norm{\nabla v}_{\oma}},
\end{equation}
where the space $H^1_*(\oma)\coloneqq \{v\in H^1(\oma),\;\int_\oma v\mathrm{d}x =0\}$ in the case of an interior vertex $\ver$, and $H^1_*(\oma)\coloneqq \{v\in H^1(\oma)\; v|_{\partial \Omega}=0\}$ in the case of a boundary vertex $\ver$.
Observe that the definition of $g_{h,\tau}^{\ver,n}$ above and the product rule $\nabla(\psia v)=v \nabla \psia + \psia\nabla v$ imply that
\begin{equation}
    (g_{h,\tau}^{\ver,n},v)_\oma - (\psia \nabla \uht,\nabla v)_\oma = (f_{h,\tau},\psia v)_\oma - (\p_t \Uht,\psia v)_{\oma} - (\nabla \uht,\nabla (\psia v))_\oma,
\end{equation}
where we furthermore notice that $\widetilde{v}=\psia v \in H^1_0(\oma)$ for either case wher $\ver$ is an interior vertex or a boundary vertex. Moreover, $\norm{\nabla (\psia v)}_{\oma}\lesssim \norm{\nabla v}_{\oma}$ for all $v\in H^1_*(\oma)$ by the Poincar\'e inequality and the inverse inequality $\norm{\nabla \psia}_{L^\infty(\oma)}\lesssim (\diam \oma)^{-1}$ by shape-regularity of $\T$.
Using the fact that $g_{h,\tau}^{\ver,n}$ and $\psi\nabla \uht$ are piecewise constant with respect to time, it is then easy to show that (c.f.\ \cite[Lemma~8.2]{ESV2017}), for any $K\in\T$ and any $\ver\in \Ver_K$,
\begin{multline}\label{eq:local_stability_main}
\left(\int_{I_n} \norm{\stha + \psia \nabla \uht}_{K}^2\mathrm{d}t \right)^{\frac{1}{2}}\leq \left(\int_{I_n} \norm{\stha + \psia \nabla \uht}_{\oma}^2\mathrm{d}t \right)^{\frac{1}{2}}
\\
\lesssim \sup_{\widetilde{v}\in \mathcal{P}_0(I_n;H^1_0(\oma))\setminus\{0\} } \frac{\int_{I_n}\left[\pair{f_{h,\tau}-f}{\widetilde{v}}_\oma+\pair{\p_t(u-\Uht)}{\widetilde{v}}_\oma+(\nabla(u-\uht),\nabla \widetilde{v})_\oma \right]\mathrm{d}t}{\left(\int_{I_n}\norm{\nabla \widetilde{v}}_\oma^2\mathrm{d}t\right)^{\frac{1}{2}}},
\end{multline}
where we have used the abbreviated notation $\pair{\cdot}{\cdot}_\oma$ to denote the duality pairing between $H^{-1}(\oma)$ and $H^1_0(\oma)$. 
The local lower bound~\eqref{eq:EY_local_lower_bound} is then obtained by applying the Cauchy--Schwarz and triangle inequalities to the terms on the right-hand side in~\eqref{eq:local_stability_main}.
The global lower bound~\eqref{eq:EY_global_lower_bound} is then deduced from~\eqref{eq:EY_local_lower_bound} by summation over all elements of the mesh.
\end{proof}

We also emphasize that the hidden constants in~\eqref{eq:EY_local_lower_bound} and \eqref{eq:EY_global_lower_bound} are, among other things, independent of the polynomial degrees $p$ and $\widetilde{p}$.

\subsection{Error bounds for the $L^2(H^1)$ norm error}

We can also consider \emph{a posteriori} error bounds for the $X$-norm errors $\norm{u-\uht}_X$ and $\norm{u-\Uht}_X$. We leave it as an exercise to the reader (or alternatively, see \cite[Theorem~5.1]{ESV2019}) to show that
\begin{equation}
\norm{u-\uht}_X \leq \left(\int_0^T\left( \norm{\sht+\nabla \uht}_\Omega^2 + \norm{\nabla(\uht-\Uht)}_{\Omega}^2\right)\mathrm{d}t\right)^{\frac{1}{2}}+\text{oscillation},
\end{equation}
for some appropriate oscillation terms in terms of $f-f_{h,\tau}$ and $u_0-u_{h,\tau,0}$. A similar bound for $\norm{u-\Uht}_X$ also holds.
Thus, up to constants and oscillation terms, the error estimator is the same as the one appearing in Theorem~\ref{thm:EY_norm_bound} above.
We note from the onset that, as explained in Section~\ref{sec:X-norm-bounds} above, we cannot expect in general to have efficiency of the estimator $\norm{\uht-\Uht}_X$ relative to the $X$-norm of the error, as this estimator corresponds to the temporal jumps of the numerical solution $\uht$.
Nevertheless, it is shown in~\cite{ESV2019} that if the mesh and time-step sizes are related by $h^2 \lesssim \tau$, then the flux estimator is bounded by the $X$-norm error plus \emph{the jump estimator}. 
Recall that the local data oscillation term $\eta_{\mathrm{osc}}^{\ver,n}$ is defined in~\eqref{eq:local_data_osc} above.

\begin{theorem}[\cite{ESV2019}]\label{thm:X_norm_lower_bounds}
For each $\ver\in\Ver$, let $h_{\oma}$ denote the diameter of the patch $\oma$.
Suppose that $1\leq \dim \leq 3$, and that there exists a constant $\gamma>0$ such that $h_{\oma}^2\leq \gamma \tau_n$ for every $\ver\in\Ver$ and every $n\in\{1,\dots,N\}$.
Then, for each $K\in\T$ and each $n\in\{1,\dots,N\}$, we have
\begin{equation}\label{eq:X_norm_local_lower_bound}
\int_{I_n}\norm{\sth+\nabla\uht}_K^2\mathrm{d}t \lesssim \sum_{\ver\in\Ver_K} \left[\int_{I_n}\left(\norm{\nabla(u-\uht)}_{\oma}^2+\norm{\nabla(\uht-\Uht)}_{\oma}^2\right)\mathrm{d}t+\left[\eta_{\mathrm{osc}}^{\ver,n}\right]^2\right],
\end{equation}
and also
\begin{equation}\label{eq:X_norm_global_lower_bound}
\int_0^T\norm{\sth+\nabla \uht}_{\Omega}^2\mathrm{d}t \lesssim \norm{u-\uht}_X^2+\norm{\uht-\Uht}_X^2 + \sum_{n=1}^N\sum_{\ver\in\Ver}\left[\eta_{\mathrm{osc}}^{\ver,n}\right]^2,
\end{equation}
where the hidden constant depends only on the shape-regularity of $\T$, the dimension $\dim$ and on $\gamma$. 
\end{theorem}

The principal consequence of~\eqref{eq:X_norm_local_lower_bound} is to show that the additional presence of the flux estimator in the fully discrete setting does not significantly alter the situation in comparison to the semi-discrete setting considered previously in Section~\ref{sec:X-norm-bounds}. 
Note that the condition $h^2\lesssim \tau$ corresponds to the case of practical computational interest since it allows for large time-steps.
We refer the reader to \cite{ESV2019} for the proof. 
Observe that the principal difference between the local lower bounds of \eqref{eq:X_norm_local_lower_bound} and \eqref{eq:EY_local_lower_bound} is that the error in the time-derivative $\p_t(u-\Uht)$ does not appear on the right-hand side of~\eqref{eq:X_norm_local_lower_bound}.

\begin{remark}[Bibliographical remarks]
 A bound of a similar nature as~\eqref{eq:X_norm_local_lower_bound} was previously obtained by Picasso in~\cite{Picasso1998} under the more restrictive two-sided condition $\tau\simeq h$.
 As mentioned already in Remark~\ref{rem:error_measure_X}, one can also consider alternative measures of the error, that combine both $\norm{u-\uht}_X$ and $\norm{u-\Uht}_X$, in which case one recovers equivalence with the estimator, see e.g.~\cite[Corollary~5.4]{ESV2019} for details.
\end{remark}

\subsection{Error bounds for the energy norm error}

Motivated by the analysis in Section~\ref{sec:hypercircle_theorem} for the semi-discrete setting, we now consider \emph{a posteriori} error bounds for the energy norm error defined by
\begin{equation}
\norm{u-\overline{u}_{h,\tau}}^2_E \coloneqq \norm{u-\overline{u}_{h,\tau}}_X^2+\frac{1}{2}\norm{u(T)-\overline{u}_{h,\tau}(T)}_\Omega^2,
\end{equation}
where the natural discrete approximation to consider is
\begin{equation}
\overline{u}_{h,\tau}\coloneqq \frac{1}{2}\left(\uht+\Uht\right).
\end{equation}
We start the analysis by proving a global efficiency bound for the jump estimator $\norm{\uht-\Uht}_X$. In order to proceed, we will make an additional assumption on the spatial discretization, namely that the $L^2$-orthogonal projection operator $\Pi_h\colon H^1_0(\Omega)\rightarrow \VTp$ is $H^1$-stable: there exists a constant $C_{\Pi}$ independent of $h$ such that
\begin{equation}\label{eq:L2_stab_ortho}
\norm{\nabla \Pi_h v}_{\Omega} \leq C_{\Pi}\norm{\nabla v}_\Omega \qquad \forall v\in H^1_0(\Omega).
\end{equation}
It is straightforward to see that \eqref{eq:L2_stab_ortho} implies the following bound: for any $w_h\in \VTp$,
\begin{equation}\label{eq:DiscreteNegnorm_stab}
\norm{w_h}_{H^{-1}(\Omega)}\leq C_{\Pi} \norm{w_h}_{\VTp^*},
\end{equation}
where $\norm{w_h}_{\VTp^*}\coloneqq \sup_{v_h\in\VTp\setminus\{0\}}\frac{(w_h,v_h)_\Omega}{\norm{\nabla v_h}_\Omega}$ is the discrete dual norm.
Note that the $H^1$-stability of the $L^2$-projection is known to hold for conforming finite element methods in a range of cases, including quasi-uniform meshes and also some graded meshes, see for instance~\cite{GaspozSiebert2016}. 
This condition is also related to the quasi-optimality of Galerkin methods for parabolic equations~\cite{TantardiniVeeser2016}.

\begin{theorem}[Efficiency of the jump estimator~\cite{Smears2025}]\label{thm:jump_discrete_energy_stab}
Suppose that~\eqref{eq:L2_stab_ortho} holds. Then,
\begin{equation}\label{eq:jump_discrete_energy_stab}
\norm{\uht-\Uht}_X \lesssim \norm{u-\overline{u}_{h,\tau}}_{E} + \wetaOscEth,
\end{equation}
where the hidden constant depends only on $C_{\Pi}$, and where the data oscillation~$\wetaOscEth$ is defined by
\begin{equation}
\wetaOscEth \coloneqq \sup_{\substack{ \varphi_h \in H^1(0,T;\VTp)\setminus\{0\} \\ \varphi_h(0)=0 }}\frac{\int_0^T\pair{f-f_{h,\tau}}{\varphi_h}\mathrm{d}t}{\normYs{\varphi_h}}.
\end{equation}
\end{theorem}

We refer the reader to~\cite{Smears2025} for the proof.
We now give the main result on the \emph{a posteriori} error bound for the energy norm of the error.

\begin{theorem}\label{thm:energy_main_bound}
Let $\sth$ be given by~\eqref{eq:sht_def}.
Then
\begin{equation}\label{eq:energy_main_upper}
\norm{u-\ouht}_E \leq \left(\int_0^T \left( \norm{\nabla(\ouht-\Uht)}_\Omega^2+\norm{\sth+\nabla\ouht}_{\Omega}^2\right)\mathrm{d}t\right)^{\frac{1}{2}} + \etaOscEth,
\end{equation}
where the data oscillation term $\etaOscEth$ is defined by
\begin{equation}
\etaOscEth \coloneqq \sup_{\varphi\in Y\setminus\{0\}}\frac{\int_0^T\pair{f-f_{h,\tau}}{\varphi} \mathrm{d}t + (u_0-u_{h,\tau,0},\varphi(0))_{\Omega}}{\normYs{\varphi}}.
\end{equation}
If $1\leq \dim \leq 3$, if~\eqref{eq:L2_stab_ortho} holds and if there exists a constant $\gamma>0$ such that $h_{\oma}^2 \leq \gamma \tau_n$ for all $\ver\in\Ver$ and all $n\in\{1,\dots,N\}$, then
\begin{equation}\label{eq:energy_main_lower}
\int_0^T \left( \norm{\nabla(\ouht-\Uht)}_\Omega^2+\norm{\sth+\nabla\ouht}_{\Omega}^2\right)\mathrm{d}t \lesssim \norm{u-\ouht}_E^2 + \left[\etaOscEthGlobal\right]^2,
\end{equation}
where
\begin{equation}
\left[\etaOscEthGlobal\right]^2\coloneqq \left[\wetaOscEth\right]^2
+ \sum_{n=1}^N\sum_{\ver\in\Ver}\left[\eta_{\mathrm{osc}}^{\ver,n}\right]^2.
\end{equation}
The hidden constant in~\eqref{eq:energy_main_lower} depends only on the shape-regularity of $\T$, the dimension $\dim$,  the constant $\gamma$, and on the constant $C_{\Pi}$ appearing in~\eqref{eq:L2_stab_ortho}.
\end{theorem}

\begin{proof}
We show here the proof of the upper bound, and refer the reader to~\cite{Smears2025} for the proof of the lower bound.
Theorem~\ref{thm:Z_infsup} shows that
\begin{equation}
\norm{u-\ouht}_E = \sup_{\varphi\in Y\setminus\{0\}}\frac{\pair{\mathcal{R}_E(\ouht)}{\varphi}_{Y^*\times Y}}{\normYs{\varphi}},
\end{equation}
where the residual $\mathcal{R}_E(\ouht)\in Y^*$ is defined by
\begin{multline}\label{eq:energy_main_1}
\pair{\mathcal{R}_E(\ouht)}{\varphi}_{Y^*\times Y}\coloneqq \int_0^T \pair{f}{\varphi}\mathrm{d}t + (u_0,\varphi(0))_\Omega
\\  - (\ouht(T),\varphi(T))_{\Omega}
-\int_0^T\left[-\pair{\p_t\varphi}{\ouht}+(\nabla \ouht,\nabla\varphi)_\Omega \right] \mathrm{d}t,
\end{multline}
for all $\varphi \in Y$.
Using the flux equilibration identity~\eqref{eq:flux_equilibration}, integration-by-parts, and the identities $\ouht(T)=\Uht(T)$ and $\Uht(0)=u_{h,\tau,0}$, we see that, for any $\varphi \in Y$,
\begin{equation}\label{eq:energy_main_2}
\int_0^T (f_{h,\tau},\varphi)_\Omega\mathrm{d}t = (\ouht(T),\varphi(T))_\Omega - (u_{h,\tau,0},\varphi(0))_{\Omega} + \int_0^T \left[-\pair{\p_t\varphi}{\Uht}-(\sth,\nabla \varphi)_\Omega \right]\mathrm{d}t.
\end{equation}
Therefore, after combining~\eqref{eq:energy_main_1} with \eqref{eq:energy_main_2}, we see that
\begin{multline}
\pair{\mathcal{R}_E(\ouht)}{\varphi}_{Y^*\times Y} = \int_0^T \pair{f-f_{h,\tau}}{\varphi}\mathrm{d}t + (u_0-u_{h,\tau,0},\varphi(0))_{\Omega}
\\ + \int_0^T\left[\pair{\p_t\varphi}{\ouht - \Uht} - (\sth+\nabla \ouht,\nabla \varphi)_{\Omega} \right]\mathrm{d}t.
\end{multline}
The Cauchy--Schwarz inequality then implies that
\begin{multline}
\sup_{\varphi\in Y\setminus\{0\}}\frac{\pair{\mathcal{R}_E(\ouht)}{\varphi}_{Y^*\times Y}}{\normYs{\varphi}} \\ \leq \left(\int_0^T \left( \norm{\nabla(\ouht-\Uht)}_\Omega^2+\norm{\sth+\nabla\ouht}_{\Omega}^2\right)\mathrm{d}t\right)^{\frac{1}{2}}+\etaOscEth,
\end{multline}
thus showing the upper bound~\eqref{eq:energy_main_upper}.
\end{proof}

Theorem~\ref{thm:energy_main_bound} represents the extension of Theorem~\ref{thm:energy_norm_bound} to the fully discrete setting, showing in particular the global efficiency of the jump and flux estimators, under the hypotheses that $h^2\lesssim \tau$ and that $L^2$-orthogonal projection $\Pi_h$ is $H^1$-stable.
Note that the lower bound for the error is global in both space and time, owing to the fact that the analysis of the jump estimator in Theorem~\ref{thm:jump_discrete_energy_stab} in \cite{Smears2025} involves a nonlocal argument in both space and time.

\section*{Conclusion}\label{sec:conclusions}

In the analysis presented above, we have demonstrated that the analytical properties of \emph{a posteriori} error estimators present a number of challenges and subtleties. In particular, the efficiency of the estimators can depend strongly on the choice of norm in which to measure the error, but also on the choice of reconstruction of the numerical solution that is considered.
Understanding the relationships between the possible choices of norms and of notion of numerical solutions is crucial for interpreting the results of the estimators in practical computations.


\begin{thebibliography}{10}

\bibitem{AdamsFournier2003}
{\sc R.~A. Adams and J.~J.~F. Fournier}, {\em Sobolev spaces}, vol.~140 of Pure
  and Applied Mathematics (Amsterdam), Elsevier/Academic Press, Amsterdam,
  second~ed., 2003.

\bibitem{AkrivisMakNoch2006}
{\sc G.~Akrivis, C.~Makridakis, and R.~H. Nochetto}, {\em A posteriori error
  estimates for the {C}rank-{N}icolson method for parabolic equations}, Math.
  Comp., 75 (2006), pp.~511--531,
  \url{https://doi.org/10.1090/S0025-5718-05-01800-4}.

\bibitem{AkrivisMakridakisNochetto2006}
{\sc G.~Akrivis, C.~Makridakis, and R.~H. Nochetto}, {\em A posteriori error
  estimates for the {C}rank-{N}icolson method for parabolic equations}, Math.
  Comp., 75 (2006), pp.~511--531,
  \url{https://doi.org/10.1090/S0025-5718-05-01800-4}.

\bibitem{AkrivisMakridakisNochetto2009}
{\sc G.~Akrivis, C.~Makridakis, and R.~H. Nochetto}, {\em Optimal order a
  posteriori error estimates for a class of {R}unge-{K}utta and {G}alerkin
  methods}, Numer. Math., 114 (2009), pp.~133--160,
  \url{https://doi.org/10.1007/s00211-009-0254-2}.

\bibitem{BergamBernardiMghazli2005}
{\sc A.~Bergam, C.~Bernardi, and Z.~Mghazli}, {\em A posteriori analysis of the
  finite element discretization of some parabolic equations}, Math. Comp., 74
  (2005), pp.~1117--1138, \url{https://doi.org/10.1090/S0025-5718-04-01697-7}.

\bibitem{BoffiBrezziFortin2013}
{\sc D.~Boffi, F.~Brezzi, and M.~Fortin}, {\em Mixed finite element methods and
  applications}, vol.~44 of Springer Series in Computational Mathematics,
  Springer, Heidelberg, 2013, \url{https://doi.org/10.1007/978-3-642-36519-5}.

\bibitem{BraessPillweinSchoberl2009}
{\sc D.~Braess, V.~Pillwein, and J.~Sch\"oberl}, {\em Equilibrated residual
  error estimates are {$p$}-robust}, Comput. Methods Appl. Mech. Engrg., 198
  (2009), pp.~1189--1197, \url{https://doi.org/10.1016/j.cma.2008.12.010}.

\bibitem{BraessSchoberl2008}
{\sc D.~Braess and J.~Sch\"oberl}, {\em Equilibrated residual error estimator
  for edge elements}, Math. Comp., 77 (2008), pp.~651--672,
  \url{https://doi.org/10.1090/S0025-5718-07-02080-7}.

\bibitem{ChenFeng2004}
{\sc Z.~Chen and J.~Feng}, {\em An adaptive finite element algorithm with
  reliable and efficient error control for linear parabolic problems}, Math.
  Comp., 73 (2004), pp.~1167--1193,
  \url{https://doi.org/10.1090/S0025-5718-04-01634-5}.

\bibitem{CostabelMcIntosh2010}
{\sc M.~Costabel and A.~McIntosh}, {\em On {B}ogovski\u i\ and regularized
  {P}oincar\'e{} integral operators for de {R}ham complexes on {L}ipschitz
  domains}, Math. Z., 265 (2010), pp.~297--320,
  \url{https://doi.org/10.1007/s00209-009-0517-8}.

\bibitem{DemlowLakkisMak2009}
{\sc A.~Demlow, O.~Lakkis, and C.~Makridakis}, {\em A posteriori error
  estimates in the maximum norm for parabolic problems}, SIAM J. Numer. Anal.,
  47 (2009), pp.~2157--2176, \url{https://doi.org/10.1137/070708792}.

\bibitem{DestuynderMetivet1999}
{\sc P.~Destuynder and B.~M\'etivet}, {\em Explicit error bounds in a
  conforming finite element method}, Math. Comp., 68 (1999), pp.~1379--1396,
  \url{https://doi.org/10.1090/S0025-5718-99-01093-5}.

\bibitem{Dupont1982}
{\sc T.~Dupont}, {\em Mesh modification for evolution equations}, Math. Comp.,
  39 (1982), pp.~85--107, \url{https://doi.org/10.2307/2007621}.

\bibitem{ErikssonJohnson1987a}
{\sc K.~Eriksson and C.~Johnson}, {\em Error estimates and automatic time step
  control for nonlinear parabolic problems. {I}}, SIAM J. Numer. Anal., 24
  (1987), pp.~12--23, \url{https://doi.org/10.1137/0724002}.

\bibitem{ErikssonJohnson1991}
{\sc K.~Eriksson and C.~Johnson}, {\em Adaptive finite element methods for
  parabolic problems. {I}. {A} linear model problem}, SIAM J. Numer. Anal., 28
  (1991), pp.~43--77, \url{https://doi.org/10.1137/0728003}.

\bibitem{ErikssonJohnson1995}
{\sc K.~Eriksson and C.~Johnson}, {\em Adaptive finite element methods for
  parabolic problems. {II}. {O}ptimal error estimates in {$L_\infty L_2$} and
  {$L_\infty L_\infty$}}, SIAM J. Numer. Anal., 32 (1995), pp.~706--740,
  \url{https://doi.org/10.1137/0732033}.

\bibitem{ErnGuermond2004}
{\sc A.~Ern and J.-L. Guermond}, {\em Theory and practice of finite elements},
  vol.~159 of Applied Mathematical Sciences, Springer-Verlag, New York, 2004,
  \url{https://doi.org/10.1007/978-1-4757-4355-5}.

\bibitem{ESV2017b}
{\sc A.~Ern, I.~Smears, and M.~Vohral\'ik}, {\em Discrete {$p$}-robust {$H({\rm
  div})$}-liftings and a posteriori estimates for elliptic problems with
  {$H^{-1}$} source terms}, Calcolo, 54 (2017), pp.~1009--1025,
  \url{https://doi.org/10.1007/s10092-017-0217-4}.

\bibitem{ESV2017}
{\sc A.~Ern, I.~Smears, and M.~Vohral\'{\i}k}, {\em Guaranteed, locally
  space-time efficient, and polynomial-degree robust a~posteriori error
  estimates for high-order discretizations of parabolic problems}, SIAM J.
  Numer. Anal., 55 (2017), pp.~2811--2834,
  \url{https://doi.org/10.1137/16M1097626}.

\bibitem{ESV2019}
{\sc A.~Ern, I.~Smears, and M.~Vohral\'ik}, {\em Equilibrated flux {\it a
  posteriori} error estimates in {$L^2(H^1)$}-norms for high-order
  discretizations of parabolic problems}, IMA J. Numer. Anal., 39 (2019),
  pp.~1158--1179, \url{https://doi.org/10.1093/imanum/dry035}.

\bibitem{ErnStephansenVohralik2010}
{\sc A.~Ern, A.~F. Stephansen, and M.~Vohral\'ik}, {\em Guaranteed and robust
  discontinuous {G}alerkin a posteriori error estimates for
  convection-diffusion-reaction problems}, J. Comput. Appl. Math., 234 (2010),
  pp.~114--130, \url{https://doi.org/10.1016/j.cam.2009.12.009}.

\bibitem{ErnVohralik2010}
{\sc A.~Ern and M.~Vohral\'ik}, {\em A posteriori error estimation based on
  potential and flux reconstruction for the heat equation}, SIAM J. Numer.
  Anal., 48 (2010), pp.~198--223, \url{https://doi.org/10.1137/090759008}.

\bibitem{ErnVohralik2020}
{\sc A.~Ern and M.~Vohral\'ik}, {\em Stable broken {$H^1$} and {$H({\rm div})$}
  polynomial extensions for polynomial-degree-robust potential and flux
  reconstruction in three space dimensions}, Math. Comp., 89 (2020),
  pp.~551--594, \url{https://doi.org/10.1090/mcom/3482}.

\bibitem{Evans1998}
{\sc L.~C. Evans}, {\em Partial differential equations}, vol.~19 of Graduate
  Studies in Mathematics, American Mathematical Society, Providence, RI, 1998,
  \url{https://doi.org/10.1090/gsm/019}.

\bibitem{GaspozSiebert2016}
{\sc F.~D. Gaspoz, C.-J. Heine, and K.~G. Siebert}, {\em Optimal grading of the
  newest vertex bisection and {$H^1$}-stability of the {$L_2$}-projection}, IMA
  J. Numer. Anal., 36 (2016), pp.~1217--1241,
  \url{https://doi.org/10.1093/imanum/drv044}.

\bibitem{GaspozSiebertKreuzerZiegler2019}
{\sc F.~D. Gaspoz, K.~Siebert, C.~Kreuzer, and D.~A. Ziegler}, {\em A
  convergent time-space adaptive {${\rm dG}(s)$} finite element method for
  parabolic problems motivated by equal error distribution}, IMA J. Numer.
  Anal., 39 (2019), pp.~650--686, \url{https://doi.org/10.1093/imanum/dry005}.

\bibitem{GeorgoulisLakkis2011}
{\sc E.~H. Georgoulis, O.~Lakkis, and J.~M. Virtanen}, {\em A posteriori error
  control for discontinuous {G}alerkin methods for parabolic problems}, SIAM J.
  Numer. Anal., 49 (2011), pp.~427--458,
  \url{https://doi.org/10.1137/080722461}.

\bibitem{JohnsonNieThomee1990}
{\sc C.~Johnson, Y.~Y. Nie, and V.~Thom\'ee}, {\em An a posteriori error
  estimate and adaptive timestep control for a backward {E}uler discretization
  of a parabolic problem}, SIAM J. Numer. Anal., 27 (1990), pp.~277--291,
  \url{https://doi.org/10.1137/0727019}.

\bibitem{Kreuzer2012}
{\sc C.~Kreuzer, C.~A. M\"oller, A.~Schmidt, and K.~G. Siebert}, {\em Design
  and convergence analysis for an adaptive discretization of the heat
  equation}, IMA J. Numer. Anal., 32 (2012), pp.~1375--1403,
  \url{https://doi.org/10.1093/imanum/drr026}.

\bibitem{KreuzerVeeser2021}
{\sc C.~Kreuzer and A.~Veeser}, {\em Oscillation in a posteriori error
  estimation}, Numer. Math., 148 (2021), pp.~43--78,
  \url{https://doi.org/10.1007/s00211-021-01194-8}.

\bibitem{LakkisMakridakis2006}
{\sc O.~Lakkis and C.~Makridakis}, {\em Elliptic reconstruction and a
  posteriori error estimates for fully discrete linear parabolic problems},
  Math. Comp., 75 (2006), pp.~1627--1658,
  \url{https://doi.org/10.1090/S0025-5718-06-01858-8}.

\bibitem{LakkisMakPryer2015}
{\sc O.~Lakkis, C.~Makridakis, and T.~Pryer}, {\em A comparison of duality and
  energy a posteriori estimates for ${L}_\infty(0,{T};{L}_2({\Omega}))$ in
  parabolic problems}, Math. Comp., 84 (2015), pp.~1537--1569,
  \url{https://doi.org/10.1090/S0025-5718-2014-02912-8}.

\bibitem{LionsMagenes1972}
{\sc J.-L. Lions and E.~Magenes}, {\em Non-homogeneous boundary value problems
  and applications. {V}ol. {I}}, vol.~Band 181 of Die Grundlehren der
  mathematischen Wissenschaften, Springer-Verlag, New York-Heidelberg, 1972.
\newblock Translated from the French by P. Kenneth.

\bibitem{LozinskiPicasso2009}
{\sc A.~Lozinski, M.~Picasso, and V.~Prachittham}, {\em An anisotropic error
  estimator for the {C}rank-{N}icolson method: application to a parabolic
  problem}, SIAM J. Sci. Comput., 31 (2009), pp.~2757--2783,
  \url{https://doi.org/10.1137/080715135}.

\bibitem{MakridakisNochetto2006}
{\sc C.~Makridakis and R.~H. Nochetto}, {\em A posteriori error analysis for
  higher order dissipative methods for evolution problems}, Numer. Math., 104
  (2006), pp.~489--514, \url{https://doi.org/10.1007/s00211-006-0013-6}.

\bibitem{NicaiseSoualem2005}
{\sc S.~Nicaise and N.~Soualem}, {\em A posteriori error estimates for a
  nonconforming finite element discretization of the heat equation}, M2AN Math.
  Model. Numer. Anal., 39 (2005), pp.~319--348,
  \url{https://doi.org/10.1051/m2an:2005009}.

\bibitem{NochettoSavareVerdi2000}
{\sc R.~H. Nochetto, G.~Savar\'{e}, and C.~Verdi}, {\em A posteriori error
  estimates for variable time-step discretizations of nonlinear evolution
  equations}, Comm. Pure Appl. Math., 53 (2000), pp.~525--589,
  \url{https://doi.org/10.1002/(SICI)1097-0312(200005)53:5<525::AID-CPA1>3.0.CO;2-M}.

\bibitem{OsborneSmears2025}
{\sc Y.~A.~P. Osborne and I.~Smears}, {\em Finite element approximation of
  time-dependent mean field games with nondifferentiable {H}amiltonians},
  Numer. Math., 157 (2025), pp.~165--211,
  \url{https://doi.org/10.1007/s00211-024-01447-2}.

\bibitem{Picasso1998}
{\sc M.~Picasso}, {\em Adaptive finite elements for a linear parabolic
  problem}, Comput. Methods Appl. Mech. Engrg., 167 (1998), pp.~223--237,
  \url{https://doi.org/10.1016/S0045-7825(98)00121-2}.

\bibitem{PragerSynge1947}
{\sc W.~Prager and J.~L. Synge}, {\em Approximations in elasticity based on the
  concept of function space}, Quart. Appl. Math., 5 (1947), pp.~241--269,
  \url{https://doi.org/10.1090/qam/25902}.

\bibitem{Roubicek2013}
{\sc T.~Roub\'{\i}\v{c}ek}, {\em Nonlinear partial differential equations with
  applications}, vol.~153 of International Series of Numerical Mathematics,
  Birkh\"{a}user/Springer Basel AG, Basel, second~ed., 2013,
  \url{https://doi.org/10.1007/978-3-0348-0513-1}.

\bibitem{Schotzau2010}
{\sc D.~Sch{\"o}tzau and T.~P. Wihler}, {\em A posteriori error estimation for
  {$hp$}-version time-stepping methods for parabolic partial differential
  equations}, Numer. Math., 115 (2010), pp.~475--509,
  \url{https://doi.org/10.1007/s00211-009-0285-8}.

\bibitem{SchwabStevenson2009}
{\sc C.~Schwab and R.~Stevenson}, {\em Space-time adaptive wavelet methods for
  parabolic evolution problems}, Math. Comp., 78 (2009), pp.~1293--1318,
  \url{https://doi.org/10.1090/S0025-5718-08-02205-9}.

\bibitem{Smears2025}
{\sc I.~Smears}, {\em On the efficiency of a posteriori error estimators for
  parabolic partial differential equations in the energy norm}, arXiv preprint 2507.13188
  (2025), \url{https://doi.org/10.48550/arXiv.2507.13188}.

\bibitem{SmearsVohralik2020}
{\sc I.~Smears and M.~Vohral\'ik}, {\em Simple and robust equilibrated flux
  {\it a posteriori} estimates for singularly perturbed reaction-diffusion
  problems}, ESAIM Math. Model. Numer. Anal., 54 (2020), pp.~1951--1973,
  \url{https://doi.org/10.1051/m2an/2020034}.

\bibitem{Sutton2020}
{\sc O.~J. Sutton}, {\em Long-time {$L^\infty(L^2)$} {\it a posteriori} error
  estimates for fully discrete parabolic problems}, IMA J. Numer. Anal., 40
  (2020), pp.~498--529, \url{https://doi.org/10.1093/imanum/dry078}.

\bibitem{TantardiniVeeser2016}
{\sc F.~Tantardini and A.~Veeser}, {\em The {$L^2$}-projection and
  quasi-optimality of {G}alerkin methods for parabolic equations}, SIAM J.
  Numer. Anal., 54 (2016), pp.~317--340,
  \url{https://doi.org/10.1137/140996811}.

\bibitem{UrbanPatera2012}
{\sc K.~Urban and A.~T. Patera}, {\em A new error bound for reduced basis
  approximation of parabolic partial differential equations}, C. R. Math. Acad.
  Sci. Paris, 350 (2012), pp.~203--207,
  \url{https://doi.org/10.1016/j.crma.2012.01.026}.

\bibitem{Verfurth1998b}
{\sc R.~Verf\"urth}, {\em A posteriori error estimates for nonlinear problems.
  ${L}^r(0,{T};{L}^\rho({\Omega}))$-error estimates for finite element
  discretizations of parabolic equations}, Math. Comp., 67 (1998),
  pp.~1335--1360, \url{https://doi.org/10.1090/S0025-5718-98-01011-4}.

\bibitem{Verfurth2003}
{\sc R.~Verf\"urth}, {\em A posteriori error estimates for finite element
  discretizations of the heat equation}, Calcolo, 40 (2003), pp.~195--212,
  \url{https://doi.org/10.1007/s10092-003-0073-2}.

\bibitem{Vohralik2010}
{\sc M.~Vohral\'ik}, {\em Unified primal formulation-based a priori and a
  posteriori error analysis of mixed finite element methods}, Math. Comp., 79
  (2010), pp.~2001--2032,
  \url{https://doi.org/10.1090/S0025-5718-2010-02375-0}.

\bibitem{Wloka1987}
{\sc J.~Wloka}, {\em Partial differential equations}, Cambridge University
  Press, Cambridge, 1987, \url{https://doi.org/10.1017/CBO9781139171755}.
\newblock Translated from the German by C. B. Thomas and M. J. Thomas.

\bibitem{Yosida1995}
{\sc K.~Yosida}, {\em Functional analysis}, Classics in Mathematics,
  Springer-Verlag, Berlin, 1995,
  \url{https://doi.org/10.1007/978-3-642-61859-8}.
\newblock Reprint of the sixth (1980) edition.

\end{thebibliography}

\end{document}